\documentclass[12pt]{amsart}
\usepackage{pgf,tikz,pgfplots} 
\pgfplotsset{compat=1.14}
\usepackage{mathrsfs}
\usetikzlibrary{arrows}
\usetikzlibrary{patterns}
\usepackage{pgfplots}
\usetikzlibrary{intersections, pgfplots.fillbetween}
\usepackage[colorlinks=true,urlcolor=blue, citecolor=red,linkcolor=blue,linktocpage,pdfpagelabels, bookmarksnumbered,bookmarksopen]{hyperref}
\usepackage[hyperpageref]{backref}
\usepackage{cleveref}
\usepackage{xcolor}
\usepackage{amsthm} 
\usepackage{latexsym,amsmath,amssymb}

\usepackage{accents}
\usepackage[colorinlistoftodos,prependcaption,textsize=tiny]{todonotes}
\usepackage{a4wide}
\usepackage{soul}
\usepackage{mathtools} 
\usepackage{xparse} 

\pgfdeclarelayer{ft}
\pgfdeclarelayer{bg}
\pgfsetlayers{bg,main,ft}

\title[Non-uniqueness and failure of Calder\'on-Zygmund below critical exponent]{Non-uniqueness and failure of Calder\'on-Zygmund estimates below the critical exponent for non-monotone PDE with linear growth}

\author{Akshara Vincent}
\address{Department of Mathematics,
University of Pittsburgh,
301 Thackeray Hall,
Pittsburgh, PA 15260, USA}
\email{akv20@pitt.edu}
\definecolor{indigo}{rgb}{0.29, 0.0, 0.51}
\definecolor{p1}{gray}{0.4}
\definecolor{p2}{gray}{0.6}
\definecolor{p3}{gray}{0.98}
\definecolor{p4}{gray}{0.8}
\definecolor{p5}{gray}{0.9}

plus 6pt minus 12pt
\def\eps{\varepsilon}

\newcommand{\m}{\mathcal{M}}

\def\B{\mathbb{B}}

\def\N{{\mathbb N}}

\renewcommand{\div}{{\rm div}}

\newtheorem{theorem}{Theorem}
\newtheorem{lemma}[theorem]{Lemma}

\newtheorem{proposition}[theorem]{Proposition}

\newtheorem{remark}[theorem]{Remark}
\newtheorem{definition}[theorem]{Definition}

\def\supp{{\rm supp\,}}

\newcommand{\R}{\mathbb{R}}

\newcommand{\brac}[1]{\left (#1 \right )}

\newcommand{\barint}{
\rule[.036in]{.12in}{.009in}\kern-.16in \displaystyle\int }

\newcommand{\barcal}{\mbox{$ \rule[.036in]{.11in}{.007in}\kern-.128in\int $}}

\def\mvint_#1{\mathchoice
          {\mathop{\vrule width 6pt height 3 pt depth -2.5pt
                  \kern -8pt \intop}\nolimits_{\kern -3pt #1}}%

          {\mathop{\vrule width 5pt height 3 pt depth -2.6pt
                  \kern -6pt \intop}\nolimits_{#1}}%
          {\mathop{\vrule width 5pt height 3 pt depth -2.6pt
                  \kern -6pt \intop}\nolimits_{#1}}%
          {\mathop{\vrule width 5pt height 3 pt depth -2.6pt
                  \kern -6pt \intop}\nolimits_{#1}}}

\numberwithin{theorem}{section} \numberwithin{equation}{section}

\newcommand{\aleq}{\precsim}
\newcommand{\ageq}{\succsim}

\let\latexchi\chi
\makeatletter
\renewcommand\chi{\@ifnextchar_\sub@chi\latexchi}
\newcommand{\sub@chi}[2]{
  \@ifnextchar^{\subsup@chi{#2}}{\latexchi_{#2}}%
}
\newcommand{\subsup@chi}[3]{
  \latexchi_{#1}^{#3}%
}
\makeatother

\makeatletter
\def\tikz@arc@opt[#1]{
  {%
    \tikzset{every arc/.try,#1}%
    \pgfkeysgetvalue{/tikz/start angle}\tikz@s
    \pgfkeysgetvalue{/tikz/end angle}\tikz@e
    \pgfkeysgetvalue{/tikz/delta angle}\tikz@d
    \ifx\tikz@s\pgfutil@empty%
      \pgfmathsetmacro\tikz@s{\tikz@e-\tikz@d}
    \else
      \ifx\tikz@e\pgfutil@empty%
        \pgfmathsetmacro\tikz@e{\tikz@s+\tikz@d}
      \fi%
    \fi
    \tikz@arc@moveto
    \xdef\pgf@marshal{\noexpand%
    \tikz@do@arc{\tikz@s}{\tikz@e}
      {\pgfkeysvalueof{/tikz/x radius}}
      {\pgfkeysvalueof{/tikz/y radius}}}%
  }%
  \pgf@marshal%
  \tikz@arcfinal%
}
\let\tikz@arc@moveto\relax
\def\tikz@arc@movetolineto#1{%
  \def\tikz@arc@moveto{\tikz@@@parse@polar{\tikz@arc@@movetolineto#1}(\tikz@s:\pgfkeysvalueof{/tikz/x radius} and \pgfkeysvalueof{/tikz/y radius})}}
\def\tikz@arc@@movetolineto#1#2{#1{\pgfpointadd{#2}{\tikz@last@position@saved}}}
\tikzset{%
  move to start/.code=\tikz@arc@movetolineto\pgfpathmoveto,%
  line to start/.code=\tikz@arc@movetolineto\pgfpathlineto}
\makeatother

\begin{document}
\begin{abstract}
We provide counterexamples to uniqueness of solutions as well as a priori Calder\'on-Zygmund estimates for solutions below $L^2$
using convex integration argument for equations of the type
$$
\div (A (\nabla u)) = 0 \quad \text{in}\; \mathbb{B}^2,
$$
where $A: \mathbb{R}^{2} \to \mathbb{R}^2$ is smooth, uniformly elliptic and has essentially linear growth,  but fails to be monotone and asymptotically Uhlenbeck.
\end{abstract}

\maketitle
\tableofcontents
\section{Introduction}
It is a classical result from linear PDE that if $A  \in C^\infty(\B^2, \R^{2 \times 2}_{sym})$ is a smooth uniformly elliptic matrix operator, that is, $\frac{1}{\Lambda} I_{2 \times 2} \leq A(x) \leq \Lambda I_{2 \times 2}$ for some $\Lambda>0$, then for any $q \in (1,\infty)$ the only distributional solution $u \in W^{1,q}_0(\B^2)$ to the equation
\[
\begin{cases}
\div (A \nabla u) = 0 \quad& \text{in $\B^2$},\\
u=0 \quad &\text{on $\partial \B^2$}
\end{cases}
\]
is $u \equiv 0$. 

This is 
intrinsically related 
to Calder\'on-Zygmund \emph{a priori} estimates for this equation. If for any $ r \in (1,\infty)$ there is some $F \in L^r(\B^2)$ and $u \in W^{1,r}_0(\B^2)$ that solve
\begin{equation}\label{eq:lineareq}
\begin{cases}
\div (A \nabla u) = \div(F) \quad& \text{in $\B^2$},\\
u=0 \quad &\text{on $\partial \B^2$,}
\end{cases}
\end{equation}
then 
\[
\int_{\B^2} |\nabla u|^r \aleq_{A,r} \int_{\B^2} |F|^r.
\]
The smoothness assumption of $A$ can be weakened to $VMO$-regularity where both of the above mentioned results still hold true.

The nonlinear version of such results does not necessarily hold. For example, for any $p \in (1,\infty)$, $p \neq 2$, the solution to
\[
\begin{cases}
\div (|\nabla u|^{p-2} \nabla u) = 0 \quad& \text{in $\B^2$},\\
u=0 \quad &\text{on $\partial \B^2$}
\end{cases}
\]
is not unique in $W^{1,r}_{0}(\B^2)$ if $r < \bar{q}$ for some $\bar{q} > \max\{1,p-1\}$. This was recently proven by Colombo-Tione \cite{colombo2022nonclassicalsolutionsplaplaceequation} disproving the Iwaniec-Sbordone conjecture \cite{IS94}. 
The corresponding Calder\'on-Zygmund result also fails: For any $r \in (\max\{1,p-1\}, \bar{q})$ and any $\Gamma > 0$ there exist maps $u \in W^{1,\infty}_0(\B^2)$, $F\in L^\infty(\B^2,\R^2)$ solving
\[
\begin{cases}
\div (|\nabla u|^{p-2} \nabla u) = \div(F) \quad& \text{in $\B^2$},\\
u=0 \quad &\text{on $\partial \B^2$,}
\end{cases}
\]
but 
\[
\int_{\B^2} |\nabla u|^{r} > \Gamma \int_{\B^2} |F|^{\frac{r}{p-1}}.
\]
This was shown in \cite{Armin24} disproving the Iwaniec conjecture \cite{I83}. Observe that while uniqueness and Calder\'on-Zygmund a priori estimates are deeply related, it is not clear whether they are equivalent for nonlinear equations.

It is natural to be interested in the more general situation, studying nonlinear equations of the type
\begin{equation}\label{eq:ourpde}
\begin{cases}
\div (A(\nabla u)) = 0 \quad& \text{in $\B^2$},\\
u=0 \quad &\text{on $\partial \B^2$}
\end{cases}
\end{equation}
where $A: \R^{2} \to \R^2$ is a uniform elliptic operator of $p$-growth 
\[
|\eta|^p \frac{1}{\Lambda} \leq \langle A(\eta), \eta\rangle \leq \Lambda |\eta|^p\;\;\;\forall \eta\in \R^2.
\]
Inspecting the arguments \cite{colombo2022nonclassicalsolutionsplaplaceequation, Armin24} we see that their arguments likely extend almost verbatim whenever $p \neq 2$, so we will focus on the case $p=2$.

We are motivated in particular by the work \cite{BDS} where they showed the following (indeed, they proved a stronger result with $x$-dependency of $A$): Let $A: \R^2 \mapsto \R^2$ satisfy the following conditions:
\begin{itemize}
    \item $A$ is continuous;
    \item $A$ has linear growth, that is, for some $\Lambda_1 > 0$
    \[
|A(\eta)|\leq \Lambda_1(1+ |\eta|)\;\;\;\forall \eta\in \R^2;
    \]
    \item $A$ is $2$-coercive, that is, for some $\Lambda_2> 0$
    \[
\Lambda_2 |\eta|^2 -\Lambda_1 \leq \langle A(\eta), \eta\rangle \;\;\;\forall \eta\in \R^2;
    \]
    \item $A$ is strictly monotone, that is,
    \[
0 < \langle A(\eta) - A(\xi), \eta-\xi \rangle \;\;\;\forall\;\eta,\xi\in \R^2, \;\eta\neq \xi;
    \]
    \item and moreover, $A$ is strongly asymptotically Uhlenbeck, that is, there exists a matrix $\Tilde{A}\in \R^{2\times 2}$ satisfying the following: For all $\varepsilon>0$ there exists $k>0$ such that for all $\eta\in \R^2$ satisfying $|\eta|\geq k$ there holds
    $$\left|\frac{\partial A(\eta)}{\partial\eta}-\Tilde{A}\right|\leq \varepsilon.$$
\end{itemize}
Under these conditions,  \cite{BDS} proved that solutions to \eqref{eq:ourpde} are unique in $W^{1,r}_0$ for any $r \in (1,\infty)$. In addition, the corresponding Calder\'on-Zygmund estimates are established under these assumptions.

In this work, we show that even for a very smooth $A: \R^2 \to \R^2$ that has all the above properties except monotonicity and Uhlenbeck condition, and is uniformly close to a linear equation, both uniqueness and Calder\'on-Zygmund may fail.

More precisely, for any $0<\theta<1$ let $A:\R^2\to \R^2$ be given as 
    \begin{equation}\label{eq: for A}
        A(\eta):= \begin{pmatrix}
        1+\sqrt{\theta}\sin(\pi \eta_1)&0\\
        0&1+\sqrt{\theta}\sin(\pi \eta_2)
    \end{pmatrix} \begin{pmatrix}
        \eta_1\\\eta_2
    \end{pmatrix}
    \end{equation}
    for $\eta=(\eta_1,\eta_2)\in \R^2$. Clearly, this $A$ is smooth, has linear growth, is $2$-coercive, and, as $\theta \to 0$, is $C^\infty$-close to the identity map $A_0(\eta) = \eta$.
It fails the monotonicity assumption and Uhlenbeck condition that would make \cite{BDS} applicable.

Given $0<\theta <1$ as above, we define the threshold exponent as 
\begin{equation}\label{eq for q bar}
    \bar{q}:= \sup_{L\in 2\N}\left(1-\frac{1}{1+L}+\frac{1}{1+L(1-\theta)}\right).
\end{equation}
Since $\theta\in (0,1)$, we readily confirm that $1<\bar{q}<2$.

Our main results are the following two theorems, addressing non-uniqueness and the failure of Calder\'on-Zygmund estimate:
\begin{theorem}[Non-uniqueness]\label{main thm uniqueness}
Let $\bar{q}$ as in \eqref{eq for q bar}.
 For any $r \in (1,\bar{q})$ there exists a non-zero function $u\in W^{1,r}_0(\B^2)$ such that
\[\begin{cases}
\div (A(\nabla u)) = 0 \quad& \text{in $\B^2$},\\
u=0 \quad &\text{on $\partial \B^2$.}
\end{cases}\]
\end{theorem}

\begin{theorem}[Failure of Calder\'on-Zygmund a priori estimate]\label{main thm CZ} 
Let $\bar{q}$ as in \eqref{eq for q bar}.
    For any $\Lambda >0$ and $r\in (1,\bar{q})$, there exist a Lipschitz map $u\in W^{1,\infty}_0(\B^2)$ and $F\in L^\infty(\B^2, \R^2)$ such that
    \[\begin{cases}
\div (A(\nabla u)) = \div F \quad& \text{in $\B^2$},\\
u=0 \quad &\text{on $\partial \B^2$.}
\end{cases}\]
    But 
    $$\int_{\B^2}|\nabla u|^r>\Lambda\left(1+ \int_{\B^2}|F|^r\right).$$
\end{theorem}

\begin{remark}
It is worth noting the asymptotic behavior of $\bar{q}$ from \eqref{eq for q bar}. As $\theta \to 0^+$, we see that $\bar{q} \to 1^+$, which is expected since the PDE \eqref{eq:ourpde} approximates the Laplace equation in this limit. 

As $\theta \to 1^-$, the ellipticity constant of $A$ tends to zero, and correspondingly $\bar{q} \to 2^-$, which is again expected since the PDE \eqref{eq:ourpde} becomes very degenerate in this case.
\end{remark}

We remark that while we focus on the two-dimensional domain, the arguments here likely extend to higher order without too much changes. However, in order to ensure the boundary datum, one might have to adapt the argument in \cite{MS2024}.

Lastly, we also comment on the regime for $r > 2$. As for Calder\'on-Zygmund, for the $p$-Laplace it is known that Calderon-Zygmund holds for $r \in (2-\eps,\infty)$. In our \emph{non-monotone} example this is not true anymore:  A solution $u$ to 
    \[\begin{cases}
\div (A(\nabla u)) = \div F \quad& \text{in $\B^2$},\\
u=g \quad &\text{on $\partial \B^2$.}
\end{cases}\]
does not necessarily satisfy 
\[
 \int_{\B^2}|\nabla u|^r \aleq_{g} \left(1+ \int_{\B^2}|F|^r\right)\;\;\;\text{for  }r>2.
\]
More precisely,
 
\begin{theorem}[Failure of Calder\'on-Zygmund a priori estimate above 2]\label{main thm CZabove2} 
There exist a $\theta \in (0,1)$, a $\bar{p} > 2$ and some linear map $l: \R^2 \to \R$ such that the following holds: 

For any $\Lambda >0$, there exist Lipschitz maps $u\in W^{1,\infty}(\B^2)$ and $F\in L^\infty(\B^2, \R^2)$ such that
    \[\begin{cases}
\div (A(\nabla u)) = \div F \quad& \text{in $\B^2$},\\
u=l \quad &\text{on $\partial \B^2$.}
\end{cases}\]
    But 
    $$\int_{\B^2}|\nabla u|^{\bar{p}}>\Lambda\left(1+ \int_{\B^2}|F|^{\bar{p}}\right).$$
\end{theorem}

{\bf Outline of the paper.}
In \Cref{section prelim}, we introduce the concept of elementary splitting  and laminates of finite order, the basic tools in convex integration. In \Cref{section CZ}, we give the proof of \Cref{main thm CZ}. 
The main idea is to use the fact that $A(\eta)$ behaves exactly linear whenever $\eta$ is an integer vector, and we leverage the non-monotonicity of $A$ by picking a sequence $b_k$ such that $\sin(\pi b_k k) \xrightarrow{k \to \infty} - \sqrt{\theta} < 0$. We then construct a laminate of finite order iterating over matrices of the form $$S_k:= \begin{pmatrix}
    b_kk & 0\\
    0& k
\end{pmatrix},$$
and the choice of $b_k$ allows us to re-weight the resulting laminate -- a similar effect to the one used in \cite{colombo2022nonclassicalsolutionsplaplaceequation,Armin24} where $b_k \equiv b$ is a constant sequence and the superlinear or sublinear growth of the equation re-weights the laminate.

In \Cref{section uniqueness}, we provide the proof of \Cref{main thm uniqueness}. As before, we first construct a laminate of finite order. This laminate is supported in four classes of matrices: the `good' class of matrices, $\mathscr{G}$ and three other classes.  From the three other classes, we construct staircase laminates that are supported in class $\mathscr{G}$. From each of these staircase laminates, we find a piecewise affine map $w:\B^2\to \R^2$ with $\Omega_{err}:= \{x\in \B^2: \nabla w\notin \mathscr{G}\}$, a set of very small measure. Using the gluing argument (\Cref{lemma gluing argument}) on the resulting maps, we then prove that $\R^{2\times 2}$ can be reduced to $\mathscr{G}$ in weak $L^q$ (see \Cref{defn R2 reduced to G}). Our approach is inspired by the presentation in \cite{KMSX24}. 

In \Cref{section CZabove2}, we present the proof of \Cref{main thm CZabove2}. The laminate construction follows the same general strategy as in \Cref{section CZ}, with the key difference that the matrix $S_k$ now carries a negative sign in one of its diagonal entries. This modification allows us to reach a threshold exponent strictly larger than 2. This is possible because the boundary datum is non-zero.

As a final remark, it appears to us that one can adapt these arguments to almost every suitably oscillating equation by choosing $b_kk$ to pick up a different phase than $k$. Indeed, as was pointed out to us by D. Faraco, this choice of following the sequence $b_k k$ is somewhat different from usual laminate approaches. Typically, the laminate follows simply a path in the matrix space, for example, it is irrelevant if one considers $b k$ or $b2^k$ or similar. Here, the weight $b_k$ is adapted to the stepsize. 
All this will be the subject of future investigation.

\begin{remark}
After completion of the manuscript, R. Tione informed us about a related result by Johansson \cite{columbostudent}, where another example for \Cref{main thm uniqueness} is given. Interestingly, it is even monotone and just fails the Uhlenbeck property.
\end{remark}

{\bf Acknowledgment.}
This project was partially funded by the National Science Foundation (NSF) grant Career DMS-2044898. The author would like to thank Armin Schikorra and Daniel Faraco for interesting discussions, and R. Tione for pointing out the work \cite{columbostudent}.

\section{Preliminaries}\label{section prelim}
We summarize a few basic definitions and preliminary arguments that will be used to prove the main results of the paper.


\subsection{Piecewise affine maps}
We say $\Omega\subset \R^2$ a \textit{regular domain} if $\Omega$ is open, bounded, connected, and the boundary $\partial \Omega$ has zero $2$-dimensional Lebesgue measure. For $X\in \R^{2\times 2}$ and $b\in \R^2$ we denote by $l_{X,b}$ the affine map $l_{X,b}(x)=Xx+b$. For the definitions given below, we follow the presentation as in \cite{KMSX24}.
\begin{definition}[Piecewise affine maps]\label[definition]{defn piecewise affine}
    Let $\Omega$ be a regular domain. We call a map $w\in W^{1,1}(\Omega,\R^2)$ piecewise affine if there exists an at most countable collection of pairwise disjoint regular domains $\Omega_i\subset \Omega$ and a null set $\mathcal{N}$ such that $\Omega=\bigcup_i \Omega_i\cup \mathcal{N}$ and $w$ is an affine map on each $\Omega_i$. That is, for each $i$, there exist $X_i\in \R^{2\times 2}$ and $b_i\in \R^2$ such that $w=l_{X_i,b_i}$ on $\Omega_i$. We also denote by $\dot{\Omega}_w=\bigcup_i\Omega_i$ (or simply $\dot{\Omega}$) the open subset of $\Omega$ where $w$ is locally affine. 
\end{definition}
\begin{definition}\label[definition]{defn R2 reduced to G}
    For $K,K'\subset \R^{2\times 2}$ and $1<p<\infty$ we say that
    $$K \text{ can be reduced to }K'\text{ in weak }L^p$$
    provided there exists a constant $C=C(K,K',p)\geq 1$ with the following property: let $M\in K$, $b\in \R^2$, $\varepsilon$, $\alpha\in (0,1)$, $s\in (1,\infty)$, and $\Omega\subset \R^2$ a regular domain. Then there exists a piecewise affine map $w\in W^{1,1}(\Omega,\R^2)\cap C^\alpha(\overline{\Omega},\R^2)$ with $w=l_{M,b}$ on $\partial\Omega$ and such that, with $\Omega_{err}:= \{x\in \Omega: \nabla w(x)\notin K'\}$ we have
    $$\int_{\Omega_{err}}(1+|\nabla w|)^sdx<\varepsilon |\Omega|,$$
    $$|\{x\in \Omega:|\nabla w(x)|>t\}|\leq C^p(1+|M|^p)|\Omega| t^{-p}\;\;\;\forall\;t>0.$$
\end{definition}
 The details  of the following elementary result can be found in \cite[Section 2.1]{KMSX24}.
\begin{lemma}[Gluing argument]\label[lemma]{lemma gluing argument}
    Let $w\in W^{1,1}(\Omega,\R^2)\cap C^\alpha(\overline{\Omega},\R^2)$ for some $\alpha\in [0,1)$, let $\{\Omega_i\}_i$ be a family of pairwise disjoint open subsets of $\Omega$, and for each $i$ let $v_i\in W^{1,1}(\Omega_i,\R^2)\cap C^\alpha(\overline{\Omega_i},\R^2)$ such that $v_i=w$ on $\partial\Omega_i$. Define 
    \[
    \Tilde{w}(x)=\begin{cases}
        v_i(x)\;\;\;x\in \Omega_i\;\text{for some }i,\\
        w(x)\;\;\;x\notin \bigcup_i\Omega_i.
    \end{cases}
    \]
    Then $\Tilde{w}\in W^{1,1}(\Omega, \R^2)\cap C^\alpha(\overline{\Omega},\R^2)$ with
    $$\Vert \Tilde{w}-w\Vert_{C^\alpha(\overline{\Omega},\R^2)}\leq 2\sup_i\Vert v_i-w\Vert_{C^\alpha(\overline{\Omega},\R^2)}.$$
\end{lemma}
\subsection{Laminates.} Let $\mathcal{M}(\R^{2\times 2})$ denote the space of (signed) Radon measures on $\R^{2\times 2}$ and let $\mathcal{P}(\R^{2\times 2})$ be the subset of probability measures. We recall the definition of laminates of finite order (see, for example, \cite{KMSX24},\cite{AFS08}). This concept was introduced by Pedregal \cite{Pedregal_1993} and was mentioned earlier by Dacorogna \cite{Dac85}.
\begin{definition}[Elementary splitting and laminates of finite order] Given probability measures $\nu,\mu\in \mathcal{P}(\R^{2\times 2})$ we say that $\mu$ is obtained from $\nu$ by \textit{elementary splitting} if $\nu $ has the form $\nu=\lambda \delta_B +(1-\lambda)\Tilde{\nu}$ for some $\Tilde{\nu}\in \mathcal{P}(\R^{2\times 2})$, there exist matrices $B_1,B_2\in \R^{2\times 2}$ and $\lambda'\in (0,1)$ such that $B=\lambda'B_1+(1-\lambda')B_2$ and rank$(B_2-B_1)=1$, and moreover
$$\mu= \lambda(\lambda'\delta_{B_1}+(1-\lambda')\delta_{B_2})+(1-\lambda)\Tilde{\nu}.$$
The set $\mathcal{L}(\R^{2\times 2})$ of laminates of finite order is defined as the smallest set which is invariant under elementary splitting and contains all  Dirac masses. 

\end{definition}
Let $\nu\in \mathcal{L}(\R^{2\times 2})$ is supported on a finite set of matrices, i.e. is of the form $\nu= \sum_{j=1}^N \lambda_j \delta_{X_j}$. The center of mass, or \textit{barycenter}, of $\nu$ will be denoted by $\overline{\nu}:= \sum_{j=1}^N \lambda_j X_j$. It is easy to see that the center of mass is invariant under splitting. We say $B_1,B_2\in \R^{2\times 2}$ are \textit{rank-one connected} if $$\text{rank}(B_1-B_2)=1.$$

We start with an important proposition about laminates of finite order. There are various versions of the following proposition (see, for example, \cite{Cel93, DM99, syc01, KSM03, DMP08, Pom10, MS03, kir03}). We refer to the proof presented in \cite[Lemma 2.2]{KMSX24}.
\begin{proposition}\label[proposition]{lemma for w from finite laminate}
    Let $\nu\in \mathcal{L}(\R^{2\times 2})$ be a laminate of finite order with center of mass $X$. Write $\nu= \sum_{j=1}^J\lambda_j \delta_{X_j}$ with $\lambda_j>0$ and $X_j\neq X_k$ for $j\neq k$. For any $b\in \R^2$, any $\varepsilon>0$ and any regular domain $\Omega\subset \R^2$ there exists a piecewise affine Lipschitz map $w:\Omega\to \R^2$ with
    \begin{itemize}
        \item[(i)] $w= l_{X,b}$ on $\partial \Omega$,
        \item[(ii)] $\Vert \nabla w\Vert_{L^\infty(\Omega)}\leq \max_{j}|X_j|$ 
        \item[(iii)] and 
        $$(1-\varepsilon)\lambda_j |\Omega| \leq |\{ x\in \Omega:\nabla w=X_j \} |\leq (1+\varepsilon) \lambda_j |\Omega|$$
        for each $j=1,...,J$.
    \end{itemize}
\end{proposition}
Now, we define \textit{staircase laminates} via the following proposition \cite[Proposition 3.1]{KMSX24}. After its introduction by Faraco \cite{faracomilton}, there are various applications of staircase laminates in the literature \cite{KSM03, CFM05, CFMM05, AFS08, bsv13, FMCO18, fls21, colombo2022nonclassicalsolutionsplaplaceequation}.
\begin{proposition}[Staircase laminates]\label[proposition]{prop defn of staricase}
    Let $K\subset \R^{2\times 2}$ and $X\notin K$. Suppose that there exists a sequence of matrices $X_n\in \R^{2\times 2}\backslash K$, $n=0,1,2,..$ with $X_0=X$, a sequence of probability measures $\mu_n\in \mathcal{P}(K)$ supported in $K$ as well as scalars $\gamma_n\in (0,1)$ such that
    \begin{itemize}
        \item[(a)] for any $n\in \N$ the probability measures 
        $$\omega_n=(1-\gamma_n)\mu_n+ \gamma_n\delta_{X_n}$$
        are laminates of finite order with barycenter $\overline{\omega_n}=X_{n-1}$;
        \item[(b)] the sequence $|X_n|$ is monotonically increasing with $\lim_{n\to \infty}|X_n|=\infty$;
        \item[(c)] $\lim_{n\to \infty}\beta_n=0$, where $$\beta_n=\prod_{k=1}^n \gamma_k,\;\;\beta_0=1.$$ 
    \end{itemize}
    For $N=1,2,...$, we define the probability measures $\nu^N$ by 
    $$\nu^N =\sum_{n=1}^N \beta_{n-1}(1-\gamma_n)\mu_n +\beta_N \delta_{X_N},$$
    by iteratively replacing $\delta_{X_{n-1}}$ by $\omega_n$ for $1\leq n\leq N$. The $\nu^N$ is a laminate of finite order with $\supp \nu^N\subset K\cup \{X_N\}$ and barycenter $\overline{\nu^N}=X$. Moreover, for any Borel set $E\subset \R^{2\times 2}$ the limit
    $$\nu^{\infty}(E)=\lim_{N\to \infty}\nu^N(E)$$
    exists and defines a probability measure $\nu^\infty$ with $\supp \nu^\infty\subset K$ and $\overline{\nu^\infty}=X$. The probability measure $\nu^\infty$ is called a staircase laminate.
\end{proposition}
Next, we have the following key proposition about staircase laminates.
The proof follows via an inductive argument by using \Cref{lemma for w from finite laminate} and is presented in \cite[Proposition 4.4]{KMSX24}. 
\begin{proposition}\label[proposition]{prop from w from nu infinity}
    Suppose $\nu^\infty$ is a staircase laminate supported in $K\subset \R^{2\times 2}$, with barycenter $M$ and satisfying the bound
    $$\nu^{\infty}(\{X:|X|>t\})\leq C (1+|M|^p)t^{-p}\;\;\;\forall \;t> 1$$ 
    for $1<p<\infty$ and for some constant $C=C(K,p)\geq 1$. Then, for each $b\in \R^2$, $\varepsilon\in (0,1)$, $\alpha\in (0,1)$, $s\in (1,\infty)$, and each regular domain $\Omega\subset \R^2$, there exists a piecewise affine map $w\in W^{1,1}(\Omega,\R^2)\cap C^{\alpha}(\overline{\Omega},\R^2)$ with $w=l_{M,b}$ on $\partial\Omega$ and the following properties: with $\Omega_{err}:= \{x\in \Omega: \nabla w(x)\notin K\}$ we have
    $$\int_{\Omega_{err}}(1+|\nabla w|)^sdx<\varepsilon|\Omega|,$$
    $(\text{and consequently, } |\Omega_{err}|<\varepsilon|\Omega|)$ and for each Borel set $E\subset \R^{2\times 2} $,
    $$(1-\varepsilon)\nu^\infty(E)\leq \frac{\{x\in \Omega\backslash \Omega_{err}: \nabla w(x)\in E \}}{|\Omega|}\leq (1+\varepsilon)\nu^\infty(E).$$
\end{proposition}
We conclude this section with the following theorem \cite[Theorem 4.1]{KMSX24}.
\begin{theorem}\label{th 4.1}
    Let $K\subset \R^{2\times 2}$ and $1<p<\infty$ such that $\R^{2\times 2}$ can be reduced to $K$ in weak $L^p$. Then for any regular domain $\Omega\subset \R^2$, any $M\in \R^{2\times 2}$, $b\in \R^2$ and any $\delta>0$, $\alpha\in (0,1)$ there exists a piecewise affine map $w\in W^{1,1}(\Omega,\R^2)\cap C^\alpha(\overline{\Omega},\R^2)$ with $w=l_{M,b}$ on $\partial\Omega$ such that 
    $$\nabla w(x)\in K\;\;a.e.\;x\in \Omega, $$
    $$\Vert w-l_{M,b}\Vert_{C^\alpha(\overline{\Omega},\R^2)}<\delta,$$
    and 
    $$|\{x \in \Omega: |\nabla w(x)|>t\}|\leq C (1+|M|^q)|\Omega| t^{-q}\;\;\;\forall t>0$$
    where $C\geq 1$ is a constant coming from \Cref{defn R2 reduced to G}.
\end{theorem}
\section{\texorpdfstring{Failure of Calder\'on-Zygmund a priori estimate: Proof of \Cref{main thm CZ}}{CZ estimate}}\label{section CZ}

We first present the proof of \Cref{main thm CZ}, as it yields a cleaner construction of the laminate and the corresponding piecewise affine map. To begin, we fix the following parameters and define the function $h$ that will be used throughout the rest of the paper.

Let $0<\theta<1$ be as in (\ref{eq: for A}) and $\bar{q}$ be as in (\ref{eq for q bar}). We define $h:\R\to \R$ by
\begin{equation}\label{eq for h}
    h(t):=\sqrt{\theta}\sin(\pi t).
\end{equation}
Fix any $q$ such that 
\begin{equation}\label{eq q less than q bar}
    1< q<\bar{q}.
\end{equation}
Then, we can find an even integer $L>0$ such that $q<q_L$ where 
\begin{equation}\label{eq for q_L}
    q_L:= 1-\frac{1}{1+L}+\frac{1}{1+L(1-\theta)}.
\end{equation}
First, we construct a sequence $\{b_k\}_{k\in \N}$ which converges to $L$ satisfying certain conditions as below.
\begin{lemma}\label[lemma]{lemma_for b_k}
    Let $L\in 2\N$. Then there exists a sequence of irrational numbers $\{b_k\}_{k\in \N}$ which satisfies the following conditions:
    \begin{itemize}
        \item[(a)] $b_k$ is monotonically increasing with $1<b_k<L$; 
        \item[(b)] $\lim\limits_{k\to \infty}b_k=L$;
        \item[(c)] $h(b_k k)\in (-\theta,-\theta^{\frac{3}{2}}) $ $\forall\;k$ and $\lim\limits_{k\to \infty}h(b_k k)= -\theta$;
        \item[(d)] $b_{k+1}-b_k=O(1/k^2)$ and $L-b_k=O(1/k)$ for all $ k$.
    \end{itemize}
\end{lemma}
\begin{proof}
    Since $\sin(\pi t)$ is a continuous function, there exists $\omega_0\in (1,\frac{3}{2})$ such that $$\sin(\omega_0 \pi)=-\sqrt{\theta}.$$
    Choose an irrational number $\delta\in (0,\omega_0-1)$ small enough so that $\sin((\omega_0-\delta)\pi)<-\theta$ and $\omega_0- \frac{\delta}{k}$ is irrational for all $k\in \N$. 
    Define a sequence of irrational numbers as follows:
    $$\delta_k:=\frac{\delta}{k}.$$
     Then,
     \begin{equation}\label{eq: h in (0,theta}
         \sin((\omega_0-\delta_k)\pi)\in \left(-\sqrt{\theta},-\theta\right)\;\;\;\forall\;k.
     \end{equation}
     Clearly, $\lim\limits_{k\to \infty}\delta_k=0$. 
    Define $n_1=n_2=0$ and for $k>2$, 
    $$n_k:= \frac{kL}{2}-2 \in \N.$$ 
    We define the sequence $\{b_k\}_{k\in \N}$ as: 
    $$b_k:= \frac{2n_k+ \omega_0-\delta_k}{k}= L- \frac{4 -\omega_0}{k}-\frac{\delta}{k^2}.$$
    Note that $4-\omega_0>0$, $\delta>0$ and
    $$b_k k= 2n_k +(\omega_0-\delta_k).$$
    Also, $b_1=\omega_0-\delta> \omega_0-(\omega_0 -1)=1$.
    Then, we observe the following about the sequence $\{b_k\}_{k\in \N}$:
    \begin{itemize}
        \item $1<b_k<L$;
        \item $b_k$ is a monotonically increasing irrational sequence that converges to $L$;
        \item $b_{k+1}-b_k=(4-\omega_0)(\frac{1}{k}-\frac{1}{k+1})+\delta (\frac{1}{k^2}-\frac{1}{(k+1)^2})= O (\frac{1}{k^2})$ for all $k$;
        \item $L-b_k= O (\frac{1}{k})$ for all $k$.
        \item $h(b_k k) = \sqrt{\theta} \sin((\omega_0 -\delta_k)\pi)\in \left(-\theta,-\theta^{\frac{3}{2}}\right)$ (follows from (\ref{eq: h in (0,theta}));
        \item $\lim\limits_{k\to \infty}h(b_k k)=\sqrt{\theta}\sin(\omega_0\pi)=-\theta$.
    \end{itemize}
    This proves the lemma.
\end{proof}

\textbf{Notation.} Let $\{b_k\}_{k\in \N}$ be as in \Cref{lemma_for b_k} with $\lim_{k\to \infty}b_k=L$. For $k\in \N$, we define 
$$S_k:= \begin{pmatrix}
    b_kk & 0\\
    0& k
\end{pmatrix},$$
$$G_{k+1}:= \begin{pmatrix}
    b_{k}k & 0\\
    0& -b_{k}k(1+h(b_k k))
\end{pmatrix},$$
and 
$$\Tilde{G}_{k}:= \begin{pmatrix}
    -k & 0\\
    0& k
\end{pmatrix}.
$$
Suppose that $w=(u,v): \B^2\to \R^2$ is a function such that $\nabla w=G_{k+1}$ or $\nabla w=\Tilde{G}_{k}$ a.e. in $\B^2$ for some $k$. Then, we note that $$A(\nabla u)-\nabla^{\perp} v=0\;\;\text{a.e. in } \B^2$$ 
where $\nabla^\perp v=\begin{pmatrix}
    -\partial_2v\\ \partial_1 v
\end{pmatrix}$.

Now, we have the proposition that gives the laminate of finite order.
\begin{proposition}\label[proposition]{prop_laminate}
    Let $q$ be as in \eqref{eq q less than q bar}, $L$ be as in \eqref{eq for q_L} and $\{b_k\}_{k\in \N}$ be as in $\Cref{lemma_for b_k}$. 
    Then for integers $N\geq k_0\geq 2$, there exist sequences $\{\overline{\alpha}_k\}_{k=k_0}^N$, $\{\overline{\beta}_k\}_{k=k_0}^N$ with values in $(0,1)$ and a number $\overline{\Gamma}_N\in (0,1)$ such that
    \begin{itemize}
        \item[(a)] the following is a laminate of finite order with barycenter $0\in \R^{2\times 2}$
        $$\sum_{k=k_0}^N(\overline{\alpha}_k \delta_{G_k}+\overline{\alpha}_k \delta_{-G_k}+\overline{\beta}_k \delta_{\Tilde{G}_k}+\overline{\beta}_k \delta_{-\Tilde{G}_k})+\overline{\Gamma}_N \delta_{S_N}+\overline{\Gamma}_N \delta_{-S_N};$$
        \item[(b)] and we have
        $$\overline{\Gamma}_N \leq C k_0^qN^{-q}$$
        for some constant $C>0$ depending only on $q$ and $L$.
    \end{itemize}
\end{proposition}
The proof of the above proposition follows from the following three lemmas. That is, we start with \Cref{initial splitting from 0}, then apply an induction argument using \Cref{lemma of splitting} and use \Cref{lemma for beta_n estimate}.
  \begin{lemma}\label[lemma]{lemma of splitting}
    For $k\in \N$, the following decomposition comes from a laminate of finite order
    $$S_k= \alpha_{k+1}G_{k+1}+\beta_{k+1}\Tilde{G}_{k+1}+\gamma_{k+1}S_{k+1}$$
    where 
    $$\alpha_{k+1}=\frac{1}{(k+1)+ b_k k(1+h(b_k k))},$$
    $$\beta_{k+1}= (1-\alpha_{k+1})\left(\frac{b_{k+1}(k+1)-b_k k}{b_{k+1}(k+1)+(k+1)}\right)$$
    and 
    $$\gamma_{k+1}=(1-\alpha_{k+1})\left(1-\frac{b_{k+1}(k+1)-b_k k}{b_{k+1}(k+1)+(k+1)}\right).$$
\end{lemma}
\begin{proof}
    For $k\geq 1$, the following is an elementary splitting of matrices:
    \[
    \begin{split}
        S_k= \begin{pmatrix}
    b_kk & 0\\
    0& k
\end{pmatrix}= &\alpha_{k+1} \begin{pmatrix}
    b_{k}k & 0\\
    0& -b_{k}k(1+h(b_k k))
\end{pmatrix}+ (1-\alpha_{k+1})\begin{pmatrix}
    b_{k}k & 0\\
    0& k+1
\end{pmatrix}\\
=&\alpha_{k+1} G_{k+1}+ (1-\alpha_{k+1})\begin{pmatrix}
    b_{k}k & 0\\
    0& k+1
\end{pmatrix}
    \end{split}
    \]
    where 
    $$\alpha_{k+1}:=\frac{1}{(k+1)+ b_k k(1+h(b_k k))}.$$
    Now, we do the following elementary splitting on the second matrix:
    \[
    \begin{split}
        \begin{pmatrix}
    b_{k}k & 0\\
    0& k+1
\end{pmatrix}&= \Tilde{\beta}_{k+1} \begin{pmatrix}
    -(k+1) & 0\\
    0& k+1
\end{pmatrix} + (1-\Tilde{\beta}_{k+1}) \begin{pmatrix}
    b_{k+1}(k+1) & 0\\
    0& k+1
\end{pmatrix}\\
&= \Tilde{\beta}_{k+1} \Tilde{G}_{k+1} + (1-\Tilde{\beta}_{k+1})S_{k+1}
    \end{split}
    \]
    where 
    $$\Tilde{\beta}_{k+1}:=\frac{b_{k+1}(k+1)-b_k k}{b_{k+1}(k+1)+(k+1)}.$$
    These two splittings give the desired result.
\end{proof}
\begin{lemma}\label[lemma]{initial splitting from 0}
    The following is a laminate of finite order with barycenter $0\in \R^{2\times 2}$:
    $$\frac{1}{4}\delta_{G_{k_0}}+\frac{1}{4}\delta_{-G_{k_0}} +\frac{1}{4}\delta_{S_{k_0}}+\frac{1}{4}\delta_{-S_{k_0}} $$
    for some $k_0\in \N$.
\end{lemma}
\begin{proof}
    The proof follows from the following two rank-one elementary splittings:
    $$\begin{pmatrix}
    0 & 0\\
    0& 0
\end{pmatrix}=\frac{1}{2}\begin{pmatrix}
    b_{k_0}k_0& 0\\
    0& 0
\end{pmatrix}+\frac{1}{2}\begin{pmatrix}
    -b_{k_0}k_0& 0\\
    0& 0
\end{pmatrix}$$
and
\[\begin{split}
    \pm\begin{pmatrix}
    b_{k_0}k_0& 0\\
    0& 0
\end{pmatrix}&=\pm \lambda \begin{pmatrix}
    b_{k_0}k_0& 0\\
    0& -b_{k_0}k_0(1+h(b_{k_0} k_0))
\end{pmatrix}+\pm (1-\lambda)\begin{pmatrix}
    b_{k_0}k_0& 0\\
    0& k_0
\end{pmatrix}\\
&= \lambda(\pm G_{k_0+1})+(1-\lambda)(\pm S_{k_0})
\end{split}\]
where 
$$\lambda := \frac{k_0}{b_{k_0}k_0(1+h(b_{k_0} k_0))+k_0}\in (0,1).$$
\end{proof}
\begin{lemma}\label[lemma]{lemma for beta_n estimate}
Let $q$ be as in \eqref{eq q less than q bar}, $L$ be as in \eqref{eq for q_L} and $\{b_k\}_{k\in \N}$ be as in $\Cref{lemma_for b_k}$. There exists a constant $C=C(q,L)>0$ such that the following holds: 
Fix $m\in \N$. For $k\in \N$, define 
    \begin{equation}
        \gamma_{k+1}^{(m)}:= \left(1- \frac{1}{m+k+1+(m+ b_k k)(1+h(b_k k))}\right)\left(1-\frac{b_{k+1}(k+1)-b_k k}{m+b_{k+1}(k+1)+m+k+1}\right).
    \end{equation}
Then
    $$\prod_{k=k_0}^N \gamma_{k+1}^{(m)}\leq C k_0^qm^q N^{-q} $$
    for all $N\geq k_0\geq1$. In particular, for $\{\gamma_{k+1}\}_{k=1}^N$ is as in \Cref{lemma of splitting} and $\overline{\Gamma}_N$ is as in \Cref{prop_laminate}, we have 
    $$\overline{\Gamma}_N\leq \prod_{k=k_0}^N \gamma_{k+1} \leq \prod_{k=k_0}^N \gamma_{k+1}^{(1)}\aleq_{q,L}k_0^qN^{-q}$$
    for all $N\geq k_0\geq 1$.
\end{lemma}
\begin{proof}
We have
$$\gamma_{k+1}^{(m)}=\left(1- \frac{1}{m+k+1+(m+ b_k k)(1+h(b_k k))}\right)\left(1-\frac{b_{k+1}(k+1)-b_k k}{m+b_{k+1}(k+1)+m+k+1}\right).$$
Then
\begin{equation}\label{eq -log gamma}
\begin{split}
    -\log \gamma_{k+1}^{(m)} =&-\log \left(1-\frac{1}{(m+k+1)+ (m+b_k k)(1+h(b_k k))}\right)\\ &-\log \left(1- \frac{b_{k+1}(k+1)-b_k k}{2m+b_{k+1}(k+1)+(k+1)}\right).
\end{split}
\end{equation}
Let 
\[
\begin{split}
    \alpha_k&:=\frac{1}{(m+k+1)+ (m+b_k k)(1+h(b_k k))}\\
    &= \frac{1}{k(1+b_k (1+h(b_k k)))\left(1+\frac{1+m(2+h(b_k k))}{k(1+b_k(1+h(b_k k)))}\right)}.
\end{split}
\]
Since $-\theta<h(b_k k)<0$ and $m>1$, we have
$$\varepsilon_1:= \frac{1+m(2+h(b_k k))}{k(1+b_k(1+h(b_k k)))} < \frac{3m}{k}.$$
For $k>3m$, we have $0<\varepsilon_1<1$. Now, we use the expansion
\[
\begin{split}
    \frac{1}{1+\varepsilon_1}&=(1-\varepsilon_1)+(\varepsilon_1^2-\varepsilon_1^3)+(\varepsilon_1^4-\varepsilon_1^5)+...\\
    &\geq  1-\varepsilon_1\\
    &\geq 1-\frac{3m}{k}.
\end{split}
\]
This implies for all $k\geq 3m$
\begin{equation}\label{eq: lower bound for alpha k}
\begin{split}
    \alpha_k&\geq \frac{1}{k(1+b_k(1+h(b_k k)))}\left( 1- \frac{3 m}{k}\right)\\&\geq \frac{1}{k(1+b_k(1+h(b_k k)))} - \frac{3 m}{k^2}.
\end{split}
\end{equation}
Now, let 
\[
\begin{split}
\Tilde{\alpha}_k&:=\frac{b_{k+1}(k+1)-b_k k}{b_{k+1}(k+1)+(k+1)+2m}\\&= \frac{b_{k+1}+k(b_{k+1}-b_k)}{k(1+b_{k+1})}\left(\frac{1}{1+\frac{1+b_{k+1}+2m}{k(1+b_{k+1})}}\right).
\end{split}
\]
We have
$$\varepsilon_2:= \frac{1+b_{k+1}+2m}{k(1+b_{k+1})}<\frac{C(L) m}{k}<1$$
for all $k>C(L) m$ for some constant $C(L)>0$. As before, we have
$$\frac{1}{1+\varepsilon_2}\geq 1-\frac{C(L) m}{k}$$
whenever $k>C(L) m$. Thus, for all $k\geq C(L)m$
\[
\begin{split}
    \Tilde{\alpha}_k&\geq \frac{b_{k+1}+k(b_{k+1}-b_k)}{k(1+b_{k+1})}\left( 1- \frac{C(L) m}{k}\right)\\
    &= \frac{1}{k}\left(\frac{b_{k+1}}{1+b_{k+1}}+\frac{k(b_{k+1}-b_k)}{1+b_{k+1}}\right)\left(1- \frac{C(L)m}{k}\right)\\
    &\geq \frac{1}{k}\left(\frac{b_{k+1}}{1+b_{k+1}}\right)\left(1- \frac{C(L)m}{k}\right)\\
    &\geq \frac{1}{k}\left(\frac{b_{k+1}}{1+b_{k+1}}\right)- \frac{C(L)m}{k^2}.
\end{split}
\]
Therefore, we get
\begin{equation}\label{eq: lower bound for alpha tilde k}
    \Tilde{\alpha}_k\geq \frac{1}{k}\left(\frac{b_{k+1}}{1+b_{k+1}}\right)- \frac{C(L)m}{k^2}\;\;\;\forall k>C(L)m.
\end{equation}
Now, using the expansion 
$$-\log (1-\varepsilon)=\varepsilon+\frac{\varepsilon^2}{2}+\frac{\varepsilon^3}{3}+.. \geq \varepsilon$$
for $0<\varepsilon<1$, in \eqref{eq -log gamma} we get
$$-\log \gamma_{k+1}^{(m)}\geq \alpha_k+\Tilde{\alpha}_k.$$
Using \eqref{eq: lower bound for alpha k} and \eqref{eq: lower bound for alpha tilde k}, we have
\[
\begin{split}
    -\log \gamma_{k+1}^{(m)}&\geq \frac{1}{k}\left(\frac{1}{1+b_k(1+h(b_k k))}+ \frac{b_{k+1}}{1+b_{k+1}}\right)-\frac{C_1m}{k^2}\\
    &= \frac{1}{k}\left(1- \frac{1}{1+b_{k+1}}+ \frac{1}{1+b_k (1+h(b_k k))} \right)-\frac{C_1m}{k^2}
\end{split}
\]
for all $k>C_1m$ for some large constant $C_1>0$ (depending only on $L$).
Recall from \eqref{eq for q_L} that
$$q_L= 1-\frac{1}{1+L}+\frac{1}{1+L(1-\theta)}$$
and $q<q_L$. Then there exists an integer $C_2>0$ (depending only on $q$ and $L$) large enough such that 
\[-\log \gamma_{k+1}^{(m)}\geq \frac{1}{k}q\;\;\forall\;k\geq C_2m.\]
So we get
\[\gamma_{k+1}^{(m)}\leq e^{-\frac{1}{k}q}\;\;\forall\;k\geq C_2m.\]
Since $$\sum_{k=C_2 m k_0}^N\frac{1}{k}\geq \log \left(\frac{N}{C_2m}\right),$$
we get
$$\prod_{k=k_0}^N \gamma_{k+1}^{(m)}\leq \prod_{k=C_2mk_0}^N \gamma_{k+1}^{(m)}\leq e^{-q\sum_{k=C_2mk_0}^N\frac{1}{k}} \leq\left(\frac{N}{C_2m k_0}\right)^{-q}\leq C(q,L)k_0^qm^qN^{-q}$$
for all $N\geq 2C_2mk_0$. Now, for $N=1,2,...,2C_2mk_0$
$$\prod_{k=k_0}^N \gamma_{k+1}^{(m)}\leq 1 \leq \left(\frac{2C_2mk_0}{N}\right)^q\leq C(q,L)k_0^qm^qN^{-q}.$$
\end{proof}
Now, we present the main theorem of this section that will be used to prove \Cref{main thm CZ}.
\begin{theorem}\label{CZ main thm }
    For any $\Lambda>0$ and $q\in (1,\bar{q})$, there exist piecewise affine Lipschitz maps $u,v\in W^{1,2}_0(\B^2)\cap \text{Lip}(\B^2)$ such that 
    $$\int_{\B^2}|\nabla u|^r>\Lambda\left(1+ \int_{\B^2}|A(\nabla u)-\nabla^{\perp }v|^r\right)$$
    for all $r\in (1,q)$ where $\nabla^\perp v=\begin{pmatrix}
        -\partial_2v \\\partial_1 v
    \end{pmatrix}$.
\end{theorem}
\begin{proof}
    Fix $\Lambda>0$ and $q\in (1,\bar{q})$. Let $N\gg1$ (to be chosen depending on $\Lambda$). We apply \Cref{prop_laminate} to get the laminate of finite order. Let $0<\varepsilon<N^{-q}$. Next, we apply \Cref{lemma for w from finite laminate} to obtain a piecewise affine Lipschitz map 
    $$w=\begin{pmatrix}
        u\\v
    \end{pmatrix}: \B^2\to \R^2$$
    such that $w=0$ on $\partial \B^2$ and $\Vert \nabla w\Vert_{L^{\infty}(\B^2)}\aleq_L N$. Now, we claim that 
    \begin{equation}\label{eq claim}
        \int_{\B^2}|\nabla u|^r>\Lambda \left(1+\int_{\B^2}|A(\nabla u)-\nabla^{\perp }v|^r\right)
    \end{equation}
    for $r\in (1,q)$.
    
    We prove the claim by method of contradiction. Assume \Cref{eq claim} is false. Then, we have 
    \begin{equation}\label{eq claim contra}
        \int_{\B^2}|\nabla u|^r\leq \Lambda \left(1+\int_{\B^2}|A(\nabla u)-\nabla^{\perp }v|^r\right).
    \end{equation}
    We denote for $S\in \R^{2\times 2}$
    $$\Omega_S:= \left\{x\in \B^2:\;\nabla w (x)=S\right\}.$$
    Define $\mathcal{S}:= \{\pm G_k, \pm \Tilde{G}_k\;: k=k_0,...,N\}$ where $k_0$ is as in \Cref{prop_laminate} and
    $$\Omega_{err}:=\Omega\backslash \left(\Omega_{S_N}\dot{\cup} \Omega_{-S_N} \dot{\cup} \dot{\bigcup_{S\in \mathcal{S}}}\Omega_S\right)$$
    where $\dot{\cup}$ denotes a disjoint union of sets. (For simplicity, we ignore the null set $\mathcal{N}$ coming from \Cref{defn piecewise affine}.) From \Cref{lemma for w from finite laminate}, 
    we get $$|\Omega_{err}|\leq \varepsilon |\B^2|$$
    and 
    \begin{equation}\label{eq: f bound on err set}
        \int_{\Omega_{err}}|A(\nabla u)-\nabla^\perp v|^r\aleq_{L,\theta} \varepsilon \Vert \nabla w\Vert^r_{L^{\infty}(\B^2)} |\B^2|\aleq_{L,\theta} \varepsilon |\B^2| N^r\aleq_{L,\theta} 1
    \end{equation}
    since $\varepsilon<N^{-q}$. We also have 
    \[
    \begin{split}
        |\Omega_{\pm S_N}|&\aleq  \overline{\Gamma}_N|\B^2|.
    \end{split}
    \]
    Fix any $S\in \mathcal{S}$.  For any $x\in \Omega_{S}$,
    \begin{equation}\label{en: f bound on omega s}
        A(\nabla u(x))-\nabla^{\perp} v(x)=0.
    \end{equation}
    Note that
    $$\pm S_N=\pm\begin{pmatrix}
        b_N N &0\\
        0 &N
    \end{pmatrix}.$$
    So, using \Cref{lemma for beta_n estimate}, we get
    \begin{equation}\label{eq f bound on omega An}
        \int_{\Omega_{\pm S_N}}|A(\nabla u(x))-\nabla^{\perp} v(x)|^r \aleq_{L,\theta} \overline{\Gamma}_N N^r \aleq k_0^qN^{r-q}.
    \end{equation}
    Hence, using (\ref{eq: f bound on err set}), (\ref{en: f bound on omega s}) and (\ref{eq f bound on omega An}) we have
    \begin{equation}\label{eq bnd on int B^2 |f|}
        \int_{\B^2} |A(\nabla u)-\nabla^\perp v|^r\aleq 1+ k_0^qN^{r-q}
    \end{equation}
    From (\ref{eq claim contra}), we get 
    \[
    \begin{split}
       \int_{\B^2}|\nabla u|^r &\leq \Lambda \left(1+\int_{\B^2}|A(\nabla u)-\nabla^{\perp }v|^r\right)\\
       & \aleq \Lambda+   \Lambda k_0^qN^{r-q}  .
    \end{split}
    \]
    This implies 
    $$\sum_{S\in \mathcal{S}}\int_{\Omega_S}|\nabla u|^r \aleq \Lambda + \Lambda k_0^q N^{r-q}.$$
    But $|\nabla u|\ageq k_0$ on $\Omega_S$. Thus,
    $$k_0^r\sum_{S\in \mathcal{S}}|\Omega_S|\aleq \Lambda + \Lambda k_0^q N^{r-q}.$$ From \Cref{lemma for w from finite laminate}, we get
    $$\sum_{S\in \mathcal{S}}|\Omega_S|= |\B^2\backslash (\Omega_{S_N}\dot{\cup}\Omega_{-S_N}\dot{\cup}\Omega_{err})|\geq|\B^2|(1-2\overline{\Gamma}_N-\varepsilon)\ageq |\B^2|(1-ck_0^qN^{-q}-N^{-q}).$$
    Therefore, we get
    $$k_0^r(1-c k_0^qN^{-q}-N^{-q})\leq C\Lambda +C\Lambda k_0^qN^{r-q}$$
    where the constants $c$ and $C$ depend only on $\theta$, $L$, $q$ and are independent of $N$. Since $r\in (1,q)$, the above inequality gives a contradiction for large $k_0$ and $N$. This proves the claim.
\end{proof}
\begin{proof}[Proof of \Cref{main thm CZ}]
Let $\bar{q}$ be as in \eqref{eq for q bar}. For any $\Lambda>0$ and $r\in (1,\bar{q})$, choose $q$ such that $r<q<\bar{q}$. We apply 
\Cref{CZ main thm } to find piecewise affine Lipschitz maps $u,v$. Take $F:=A(\nabla u)-\nabla^\perp v$. Then, we have 
\[
\begin{cases}
\div (A(\nabla u)) = \div F \quad& \text{in $\B^2$},\\
u=0 \quad &\text{on $\partial \B^2$}.
\end{cases}
\]
Since $u$ and $v$ are piecewise affine Lipschitz maps, it follows that $u\in W^{1,\infty}_0(\B^2)$ and $F\in L^\infty(\B^2,\R^2)$. But we have
$$\int_{\B^2}|\nabla u|^r>\Lambda\left(1+ \int_{\B^2}|F|^r\right).$$ This completes the proof.
\end{proof}
\section{\texorpdfstring{Non-uniqueness: Proof of \Cref{main thm uniqueness}}{non uniqueness}}\label{section uniqueness}
\textbf{Notation.} Let $h$ be as in \eqref{eq for h}, $ q$ as in \eqref{eq q less than q bar} and $L$ as in \eqref{eq for q_L}. Let $m\in 2\N$. Let $\{b_k\}_{k\in \N}$ be as in \Cref{lemma_for b_k} with $\lim_{k\to \infty}b_k=L$. Take $b_0:=1$. For $k\in \N\cup \{0\}$, we define the following classes of matrices 
$$\mathscr{G}:= \left\{\begin{pmatrix}
    x_1&x_2\\
    x_3&x_4
\end{pmatrix}\in \R^{2\times 2}:\;x_1,x_2\neq 0, \;A\begin{pmatrix}
    x_1\\x_2
\end{pmatrix}-J\begin{pmatrix}
    x_3\\x_4 
\end{pmatrix}=0\right\} \text{ where }J:= \begin{pmatrix}
    0&-1\\1&0
\end{pmatrix},$$
    $$ \mathscr{B}_{k}^{(m)}:=\left\{ \begin{pmatrix}
        \sigma_1(m + b_k k)& \sigma_2 m\\
        \sigma_2 m& \sigma_1(m+k)
    \end{pmatrix}
    :\;\sigma_1,\sigma_2\in \{1,-1\}\right\},$$
    $$\mathscr{C}_{k}^{(m)}:= \left\{ \begin{pmatrix}
        \sigma_1 m& \sigma_2(m+b_k k)\\
        -\sigma_2(m+k)& -\sigma_1 m
    \end{pmatrix}
    :\;\sigma_1,\sigma_2\in \{1,-1\}
    \right\},$$
    and 
    $$\mathscr{D}_{k}^{(m)}:=\left\{ \begin{pmatrix}
        \sigma_1(m+b_k k)& \sigma_2 (m+b_k k)\\
        -\sigma_2(m+  k)& \sigma_1(m+k)
    \end{pmatrix}: \;\sigma_1,\sigma_2\in \{1,-1\}\right\}.$$
It is a `good' criteria to have opposite signs for the diagonal elements and same sign for off-diagonal elements. This will be more clear from \Cref{lemma for M splitting}. In the above classes, $\mathscr{B}_{k}^{(m)}, \mathscr{C}_{k}^{(m)}, \mathscr{D}_{k}^{(m)}$ don't satisfy this criteria in diagonal, off-diagonal or both, respectively. Also, note that the above classes are mutually disjoint since $\{b_k\}$ is an irrational sequence.

Given any arbitrary matrix $M\in \R^{2\times 2}$, we do an initial splitting of the matrix to end up in matrices in the above classes. Precisely, we have the following lemma which provide an initial decomposition of matrix $M$.
\begin{lemma}\label[lemma]{lemma for M splitting}
Let $M=\begin{pmatrix}
    x_1&x_2\\
    x_3&x_4
\end{pmatrix}$ where $x_1,x_2,x_3,x_4\in \R$. The following is a laminate of finite order with barycenter $M$:
$$\sum_{j=1}^4 g_j \delta_{G_0^j} + \sum_{j=1}^4b_j \delta_{B_0^j} +\sum_{j=1}^4 c_j\delta_{C_0^j} +\sum_{j=1}^4 d_j\delta_{D_0^j} $$
where $g_j,b_j,c_j,d_j\in (0,1)$, $G^j_0\in \mathscr{G}$, $B^j_0\in \mathscr{B}_{0}^{(m)}$, $C^j_0\in \mathscr{C}_{0}^{(m)}$, and $D^j_0\in \mathscr{D}_{0}^{(m)}$, for $j=1,2,3,4$.
\end{lemma}
\begin{proof}
    We pick the smallest even integer $m>0$ such that $\max\{|x_1|,|x_2|,|x_3|,|x_4|\}<m$. Clearly,
    \begin{equation}\label{eq m lesss than 1+M}
        m\leq 2(1+|M|).
    \end{equation}
    Then, we can write $x\in \{x_1,x_2,x_3,x_4\}$ as a convex combination of $-m$ and $m$ as below
    $$x=\lambda_x (-m)+(1-\lambda_x)m$$
    where 
    $$\lambda_x=\frac{m-x}{2m}\in (0,1).$$
    Now, we decompose the matrix $M$ as below
    $$M= \lambda_{x_1} \begin{pmatrix}
        -m&x_2\\
        x_3& x_4 
    \end{pmatrix}+(1-\lambda_{x_1})\begin{pmatrix}
        m&x_2\\
        x_3& x_4 
    \end{pmatrix}.$$
    Note that the matrices on the right-hand side of the above equation are rank-one connected. We decompose term $x_2$ on both the matrices above. Then we repeat this on term $x_3$ and then on $x_4$ in all the resulting matrices from the previous steps. Thus, we can write 
    $$M=\sum_{j=1}^4 g_j G_0^j + \sum_{j=1}^4b_j B_0^j +\sum_{j=1}^4 c_jC_0^j +\sum_{j=1}^4 d_jD_0^j $$
    where $g_j,b_j,c_j,d_j\in (0,1)$,  
    $$G^j_0\in\left\{ \begin{pmatrix}
        -m& m\\
        m& m
    \end{pmatrix},
    \begin{pmatrix}
        m& m\\
        m& -m
    \end{pmatrix},
    \begin{pmatrix}
        -m& -m\\
        -m& m
    \end{pmatrix},
    \begin{pmatrix}
        m& -m\\
        -m& -m
    \end{pmatrix}\right\}\subset \mathscr{G},$$
    and 
    \begin{equation}\label{eq for S 0 matrices}
        \begin{split}
            B^j_0&\in \left\{ \begin{pmatrix}
        m& m\\
        m& m
    \end{pmatrix},
    \begin{pmatrix}
        m& -m\\
        -m& m
    \end{pmatrix},
    \begin{pmatrix}
        -m& m\\
        m& -m
    \end{pmatrix},
    \begin{pmatrix}
        -m& -m\\
        -m& -m
    \end{pmatrix}\right\}= \mathscr{B}_{0}^{(m)};\\
    C^j_0&\in \left\{ \begin{pmatrix}
        -m& -m\\
        m& m
    \end{pmatrix},
    \begin{pmatrix}
        -m& m\\
        -m& m
    \end{pmatrix},
    \begin{pmatrix}
        m& -m\\
        m& -m
    \end{pmatrix},
    \begin{pmatrix}
        m& m\\
        -m& -m
    \end{pmatrix}\right\}= \mathscr{C}_{0}^{(m)};\\
    D^j_0&\in \left\{ \begin{pmatrix}
        m& -m\\
        m& m
    \end{pmatrix},
    \begin{pmatrix}
        m& m\\
        -m& m
    \end{pmatrix},
    \begin{pmatrix}
        -m& -m\\
        m& -m
    \end{pmatrix},
    \begin{pmatrix}
        -m& m\\
        -m& -m
    \end{pmatrix}\right\}= \mathscr{D}_{0}^{(m)}
        \end{split}
    \end{equation}
    for $j=1,2,3,4$. Indeed, take any $G_0^j$, say
    $$\begin{pmatrix}
        -m&m\\m&m
    \end{pmatrix}.$$
    We see that 
    $$A\begin{pmatrix}
        -m\\m
    \end{pmatrix}-J\begin{pmatrix}
        m\\m
    \end{pmatrix}= \begin{pmatrix}
        -m\\m
    \end{pmatrix}-\begin{pmatrix}
        -m\\m
    \end{pmatrix}=0$$
    and hence, $G_0^j\in \mathscr{G}$. 
    This gives the desired laminate of finite order.
    \end{proof}
    
    
    In the next lemma, we give a decomposition of any matrix $S_k$ in $\mathscr{B}_k^{(m)}$, $\mathscr{C}_k^{(m)}$ or $\mathscr{D}_k^{(m)}$ into matrices in $\mathscr{B}_{k+1}^{(m)}$, $\mathscr{C}_{k+1}^{(m)}$ or $\mathscr{D}_{k+1}^{(m)}$ and $\mathscr{G}$.
    \begin{lemma}\label[lemma]{lemma for spltting class k to k+1}
        Let $k\in \N\cup \{0\}$. Let $S_k$ be a matrix in $\mathscr{B}_{k}^{(m)}$, $\mathscr{C}_{k}^{(m)}$ or $\mathscr{D}_{k}^{(m)}$. The following decomposition comes from a laminate of finite order:
        $$S_k=\lambda_{k+1,1}^{(m)}G_{k+1,1}^{(S_k)}+ \lambda_{k+1,2}^{(m)}G_{k+1,2}^{(S_k)}+\gamma_{k+1}^{(m)}S_{k+1}$$
        where 
        \[\begin{split}
            &\lambda_{k+1,1}^{(m)}:= \frac{1}{m+k+1+ (m+b_k k)(1+h(b_k k))},\\
            &\lambda_{k+1,2}^{(m)}:= \left(1-\frac{1}{m+k+1+ (m+b_k k)(1+h(b_k k))}\right) \frac{b_{k+1}(k+1)-b_k k}{m+k+1+ m+ b_{k+1}(k+1)},\\
            &\gamma_{k+1}^{(m)}:= \left(1-\frac{1}{m+k+1+ (m+b_k k)(1+h(b_k k))}\right) \left(1-\frac{b_{k+1}(k+1)-b_k k}{m+k+1+ m+ b_{k+1}(k+1)}\right),
        \end{split}\]
        $G_{k+1,1}^{(S_k)},G_{k+1,2}^{(S_k)}\in \mathscr{G}$ and $S_{k+1}$ belongs to $\mathscr{B}_{k+1}^{(m)}$, $\mathscr{C}_{k+1}^{(m)}$ or $\mathscr{D}_{k+1}^{(m)}$ respectively. Also, $S_k\notin \mathscr{G}$, $|G_{k+1,1}^{(S_k)}|<|S_{k+1}|$, $|G_{k+1,2}^{(S_k)}|<|S_{k+1}|$ and 
        $$|S_k|\leq |S_{k+1}|\leq 2L|S_k|.$$
    \end{lemma}
    \begin{proof}
    It is clear that the above classes of matrices and the coefficients in the matrix decomposition above depend on the starting matrix $M$. But, for simplicity of notation, we drop the superscript $(m)$ (assuming the dependency is clear for given $M$).
    
    \textit{Case 1}: Let $k\in \N\cup \{0\}$ and $S_k\in \mathscr{B}_k$. That is, for some $\sigma_1,\sigma_2\in \{1,-1\}$  
    $$S_k= \begin{pmatrix}
        \sigma_1(m + b_k k)& \sigma_2 m\\
        \sigma_2 m& \sigma_1(m+k)
    \end{pmatrix}.$$
    Since $m\in 2\N$, we have
    \[
    \begin{split}
        A\begin{pmatrix}
            \sigma_1(m+b_k k)\\
            \sigma_2 m
        \end{pmatrix}- J \begin{pmatrix}
            \sigma_2 m\\ \sigma_1 (m+k)\end{pmatrix}&= \begin{pmatrix}
            \sigma_1(m+b_k k)(1+\sigma_1h(b_k k))\\
            \sigma_2 m
        \end{pmatrix}-\begin{pmatrix}
             -\sigma_1 (m+k)\\ \sigma_2 m\end{pmatrix}\\
             &= \begin{pmatrix}
            \sigma_1((m+b_k k)(1+\sigma_1h( b_k k))+m+k)\\
            0
        \end{pmatrix}\neq 0.
    \end{split}
    \]
    So, $S_k\notin \mathscr{G}$. We notice that 
    $$m+k\in   \big(-(m+b_k k)(1+h(b_k k)), \;m+k+1\big).$$ 
    So, we do the following decomposition such that the resulting matrices are rank-one connected:
    \begin{equation}\label{eq split 1 S k in B k}
    \begin{split}
    S_k=&\begin{pmatrix}
        \sigma_1(m + b_k k)& \sigma_2 m\\
        \sigma_2 m& \sigma_1(m+k)
    \end{pmatrix}\\=&\frac{1}{m+k+1+ (m+b_k k)(1+h(b_k k))}\begin{pmatrix}
        \sigma_1(m+ b_k k)& \sigma_2m\\
        \sigma_2 m& -\sigma_1(m+b_k k)(1+h(b_k k))\end{pmatrix} \\&+ \left(1-\frac{1}{m+k+1+ (m+b_k k)(1+h(b_k k))}\right)
        \begin{pmatrix}
                \sigma_1(m+ b_k k)& \sigma_2 m\\
        \sigma_2 m& \sigma_1(m+k+1)
    \end{pmatrix}. \end{split}
    \end{equation}
    We set 
    \[ \begin{split}
        \lambda_{k+1,1}&:= \frac{1}{m+k+1+ (m+b_k k)(1+h(b_k k))};\\
        G_{k+1,1}^{(S_k)}&:= \begin{pmatrix}
        \sigma_1(m+ b_k k)& \sigma_2m\\
        \sigma_2m& -\sigma_1(m+b_k k)(1+h(b_k k))\end{pmatrix}\in \mathscr{G}.
    \end{split} \]
    Again, we have
    $$m+b_k k\in \big(-(m+k+1), \;m+b_{k+1}(k+1)\big).$$
    We split second matrix on the right hand side of \eqref{eq split 1 S k in B k} into rank-one connected matrices.
    \[\begin{split}
       &\begin{pmatrix}
                \sigma_1(m+ b_k k)& \sigma_2 m\\
        \sigma_2m& \sigma_1(m+k+1)
    \end{pmatrix}\\&= \frac{b_{k+1}(k+1)-b_k k}{m+k+1+ m+ b_{k+1}(k+1)} \begin{pmatrix}
                -\sigma_1(m+ k+1)& \sigma_2m\\
        \sigma_2m& \sigma_1(m+k+1 )
    \end{pmatrix}\\&+ \left(1- \frac{b_{k+1}(k+1)-b_k k}{m+k+1+ m+ b_{k+1}(k+1)}\right)
    \underbrace
    {\begin{pmatrix}
                \sigma_1(m+ b_{k+1}(k+1)) & \sigma_2m\\
        \sigma_2m& \sigma_1(m+k+1)
    \end{pmatrix}}_{=S_{k+1}}.
    \end{split}
    \]
    We set 
    \[
    \begin{split}
        &\lambda_{k+1,2}:= \left(1-\frac{1}{m+k+1+ (m+b_k k)(1+h(b_k k))}\right) \frac{b_{k+1}(k+1)-b_k k}{m+k+1+ m+ b_{k+1}(k+1)};\\
        &G_{k+1,2}^{(S_k)}:= \begin{pmatrix}
                -\sigma_1(m+ k+1)& \sigma_2m\\
        \sigma_2m& \sigma_1(m+k+1) 
    \end{pmatrix}\in \mathscr{G};\\
    &\gamma_{k+1}:= \left(1-\frac{1}{m+k+1+ (m+b_k k)(1+h(b_k k))}\right) \left(1-\frac{b_{k+1}(k+1)-b_k k}{m+k+1+ m+ b_{k+1}(k+1)}\right).
    \end{split}
    \]
    
    \textit{Case 2}: Let $S_k \in \mathscr{C}_k$, that is
    $$S_k=\begin{pmatrix}
        \sigma_1 m& \sigma_2(m+b_k k)\\
        -\sigma_2(m+k)& -\sigma_1 m
    \end{pmatrix}$$
    for some $\sigma_1,\sigma_2\in \{1,-1\}$. We see that
    \[
    \begin{split}
        A\begin{pmatrix}
            \sigma_1m\\\sigma_2(m+b_k k)
        \end{pmatrix}- J \begin{pmatrix}
             -\sigma_2 (m+k)\\ -\sigma_1m\end{pmatrix}&= \begin{pmatrix}
             \sigma_1 m\\
            \sigma_2(m+b_k k)(1+\sigma_2h( b_k k))
            \end{pmatrix}-\begin{pmatrix}
            \sigma_1 m\\
             -\sigma_2 (m+k) \end{pmatrix}\\
             &= \begin{pmatrix}
             0\\
            \sigma_2((m+b_k k)(1+\sigma_2h( b_k k))+m+k)
        \end{pmatrix}\neq 0.
    \end{split}
    \]
    So, $S_k\notin \mathscr{G}$.
    Now, we do two similar splittings into rank-one connected matrices as in the \textit{Case 1}, but now on the off-diagonal elements,  to get the desired result.

    \textit{Case 3}: For $k\in \N\cup \{0\}$, let $S_k\in \mathscr{D}_k$. That is, for some $\sigma_1,\sigma_2\in \{1,-1\}$  
    $$S_k= \begin{pmatrix}
        \sigma_1(m+b_k k) &\sigma_2(m+b_k k)\\
        -\sigma_2(m+ k)&\sigma_1(m+ k)
    \end{pmatrix}.$$
    As in the previous cases, it can be easily verified that $S_k\notin \mathscr{G}$. Next, we do the following two splittings into rank-one connected matrices:
    \[
    \begin{split}
        S_k=&\begin{pmatrix}
        \sigma_1(m+b_k k) &\sigma_2(m+b_k k)\\
        -\sigma_2(m+ k)&\sigma_1(m+ k)
    \end{pmatrix} \\&= \lambda_{k+1,1} \begin{pmatrix}
        \sigma_1(m+b_k k) &\sigma_2(m+b_k k)\\
        \sigma_2(m+ b_kk)(1+h(b_k k))&-\sigma_1(m+ b_k k)(1+h(b_k k))
    \end{pmatrix}\\& +
    (1-\lambda_{k+1,1}) \begin{pmatrix}
        \sigma_1(m+b_k k) &\sigma_2(m+b_k k)\\
        -\sigma_2(m+ k+1)&\sigma_1(m+ k+1)
    \end{pmatrix}
    \end{split}
    \]
    and 
    \[
    \begin{split}
        (1-\lambda_{k+1,1}) &\begin{pmatrix}
        \sigma_1(m+b_k k) &\sigma_2(m+b_k k)\\
        -\sigma_2(m+ k+1)&\sigma_1(m+ k+1)
    \end{pmatrix}\\&= \lambda_{k+1,2}
    \begin{pmatrix}
        -\sigma_1(m+ k+1) &-\sigma_2(m+k+1)\\
        -\sigma_2(m+ k+1)&\sigma_1(m+ k+1)
    \end{pmatrix}\\
    &+ \gamma_{k+1} \begin{pmatrix}
        \sigma_1(m+b_{k+1} (k+1)) &\sigma_2(m+b_{k+1} (k+1))\\
        -\sigma_2(m+ k+1)&\sigma_1(m+ k+1)
    \end{pmatrix}.
    \end{split}
    \]
    We have
    \[
    \begin{split}
        G_{k+1,1}^{(S_k)}&:= \begin{pmatrix}
        \sigma_1(m+b_k k) &\sigma_2(m+b_k k)\\
        \sigma_2(m+ b_kk)(1+h(b_k k))&-\sigma_1(m+ b_k k)(1+h(b_k k))
    \end{pmatrix}\in \mathscr{G};\\
    G_{k+1,2}^{(S_k)}&:= \begin{pmatrix}
        -\sigma_1(m+ k+1) &-\sigma_2(m+k+1)\\
        -\sigma_2(m+ k+1)&\sigma_1(m+ k+1)
    \end{pmatrix}\in \mathscr{G}.
    \end{split}
    \]
    
    By \Cref{lemma_for b_k} (a), we get that $1<\frac{b_{k+1}}{b_k}<L$. Therefore, we get 
    $$|S_k|\leq |S_{k+1}|\leq 2L|S_k|$$
    in all the above cases. Since $1+h(b_k k)\in (0,1)$, we also have $|G_{k+1,1}^{(S_k)}|<|S_{k+1}|$ and $|G_{k+1,2}^{(S_k)}|<|S_{k+1}|$. This completes the proof.
    \end{proof}
Next, for any matrix $S_0$ in $\mathscr{B}_0^{(m)}$, $\mathscr{C}_0^{(m)}$ or $\mathscr{D}_0^{(m)}$ we construct staircase laminate supported in $\mathscr{G}$ with barycenter $S_0$. 
\begin{proposition}\label[proposition]{prop for staircase laminate for S_0}
    Let $S_0\in \mathscr{B}_0^{(m)}\cup \mathscr{C}_0^{(m)} \cup \mathscr{D}_0^{(m)}$. Then there exists a staircase laminate $\nu^{\infty}_{S_0}$ supported in $\mathscr{G}$ with barycenter $S_0$. Moreover, there exists a constant $C=C(q,L)>1$ such that
    $$\nu^\infty_{S_0}(\{X: |X|>t\})\leq C |S_0|^qt^{-q}\;\;\forall \;t>0$$
    where $q$ is as in \eqref{eq q less than q bar}. 
\end{proposition}
\begin{proof}
    Fix a matrix $S_0$ given as in \eqref{eq for S 0 matrices}. Clearly, $S_0\notin \mathscr{G}$. Using \Cref{lemma for spltting class k to k+1}, we find a sequence of matrices $S_n$, $n=0,1,2,..$ starting at $S_0$. We also have
    \begin{equation}\label{eq 3.4}
        |S_n|\leq |S_{n+1}|\leq 2L |S_n|.
    \end{equation}
Clearly, the sequence $|S_n|$ is monotone increasing with $\lim_{n\to \infty}|S_n|=\infty$. For $n\in \N$, we define a sequence of probability measures $\mu_n$ supported in $\mathscr{G}$ as follows: 
    \[
\begin{split}
 \mu_{n} :=& \frac{\lambda_{n,1}}{\lambda_{n,1}+\lambda_{n,2}} \delta_{G_{n,1}^{S_{n-1}}} + \frac{\lambda_{n,2}}{\lambda_{n,1}+\lambda_{n,2}} \delta_{G_{n,2}^{S_{n-1}}}\\
 = & \frac{\lambda_{n,1}}{1-\gamma_n} \delta_{G_{n,1}^{S_{n-1}}} + \frac{\lambda_{n,2}}{1-\gamma_n} \delta_{G_{n,2}^{S_{n-1}}}
\end{split}
 \]
where $\lambda_{n,1},\lambda_{n,2},\gamma_n, G^{S_{n-1}}_{n,1}$ and $G^{S_{n-1}}_{n,2}$ are as in \Cref{lemma for spltting class k to k+1} (we drop the superscript $(m)$ for simplicity of notation). Then,
\[
 \omega_{n} := (1-\gamma_n) \mu_{n} + \gamma_n \delta_{S_n} 
\]
is by our construction a laminate of finite order with barycenter $S_{n-1}$.

Set
\[
 \beta_n := \prod_{k=1}^n \gamma_k, \quad \beta_0 := 1
\]
and set
\[
\nu^N_{S_0} := \sum_{n=1}^N \beta_{n-1} (1-\gamma_n) \mu_{n} + \beta_N \delta_{S_N}.
\]
for $N\in \N$. Using \Cref{lemma for beta_n estimate}, we get $\lim_{n\to \infty}\beta_n=0$. By \Cref{prop defn of staricase}, $\nu_{S_0}^N$ is a laminate of finite order with $\supp \nu^N_{S_0} \subset \mathscr{G} \cup \{S_N\}$, and its limit is a staircase laminate $\nu^\infty_{S_0}$ with $\supp \nu^\infty \subset \mathscr{G}$ and barycenter $S_0$. Moreover, for any fixed $t > 0$ we have
\begin{equation}\label{eq lim of nu N}
    \nu^\infty_{S_0}\brac{\{X: |X| > t\}} = \lim_{N \to \infty} \nu^N_{S_0}\brac{\{X: |X| > t\}}.
\end{equation}
For $n\in \N$, we have using \Cref{lemma for spltting class k to k+1}
\begin{equation}\label{eq 3.5 a}
    \supp \mu_{n} =\{G_{n,1}^{S_{n-1}},G_{n,2}^{S_{n-1}} \}\subset \{X\in \R^{2\times 2}: |X|\leq |S_n|\}
\end{equation}
and using \Cref{lemma for beta_n estimate}, we get
\begin{equation}\label{eq 3.5 b}
    \beta_n |S_n|^q\aleq_{q,L} m^q n^{-q}(m+n)^q \aleq_{q,L} m^q\;\;\;\forall n\geq m.
\end{equation}
For $n\in \N$, let $t_n=|S_n|$. Using \eqref{eq 3.5 a} we observe that $\mu_{k}(\{X: |X|>t_n\})=0$ for all $k\leq n$, and hence, for any $N\geq n$
\[
 \nu^N_{S_0}\brac{\{X: |X| > t_n\}} \leq \sum_{k=n+1}^N \beta_{k-1} (1-\gamma_k) + \beta_N.
\]
In the above inequality, we have 
\[
\beta_{k-1} (1-\gamma_k) = \beta_{k-1} - \beta_k
\]
and thus, the sum is telescoping and we get, using in addition \eqref{eq 3.5 b},
\[
 \nu^N_{S_0}\brac{\{X: |X| > t_n\}} \leq \beta_n\aleq_{q,L}m^q |S_n|^{-q}= m^q t_n^{-q}
\]
for all $n\geq m$. Letting $N\to \infty$ we get the estimate for $t_n$, $n= m,m+1,..$. More generally, for any $t\geq t_{m}$ choose $n\geq m$ so that $t_n\leq t<t_{n+1}$. Using \eqref{eq 3.4}, we get
\[
 \nu^N_{S_0}\brac{\{X: |X| > t\}} \aleq_{q,L}m^q  t_n^{-q}\aleq_{q,L}m^qt^{-q}.
\]
For $t<t_{m}$, we have
$$\nu^N_{S_0}\brac{\{X: |X| > t\}}\leq 1\aleq_{q,L}\frac{m^q}{t_{m}^q} \aleq_{q,L}\frac{m^q}{t^q}.$$
Thus, we get from \eqref{eq lim of nu N}
$$\nu^\infty_{S_0}(\{X: |X|>t\})\aleq_{q,L} m^qt^{-q}\;\;\forall \;t>0.$$
Since, $|S_0|=m$, we conclude.
\end{proof}
Now, we combine \Cref{lemma for M splitting} with \Cref{prop for staircase laminate for S_0} and \Cref{prop from w from nu infinity} to find a piecewise affine function. Indeed, the following theorem proves that $\R^{2\times 2}$ can be reduced to $\mathscr{G}$ in weak $L^q$ (see \Cref{defn R2 reduced to G}).
\begin{theorem}\label{th matrices space reduced to G}
    There exists a constant $C=C(q,L)\geq 1$ with the following property: let $M \in \R^{2\times 2}$, $b\in \R^2$, $\varepsilon$, $\alpha\in (0,1)$, $s\in (1,\infty)$, and $\Omega\subset \R^2$ a regular domain. Then there exists a piecewise affine map $w\in W^{1,1}(\Omega,\R^2)\cap C^\alpha(\overline{\Omega},\R^2)$ such that
    \begin{itemize}
        \item[(i)] $w= M x +b$ on $\partial \Omega$;
        \item[(ii)] for $\Omega_{err}:= \{x\in \Omega: \nabla w \notin \mathscr{G}\}$, $|\Omega_{err}|<\varepsilon |\Omega|$;
        \item[(iii)] we have
        $$\int_{\Omega_{err}}(1+|\nabla w|)^s<\varepsilon |\Omega|;$$
        \item[(iv)]
    \[
        |\{x \in \Omega: |\nabla w(x)|>t\}|\leq C (1+|M|^q)|\Omega| t^{-q}\;\;\;\forall t>0.
    \]
    \end{itemize}
\end{theorem}
\begin{proof}
    Given $M\in \R^{2\times 2}$, we apply \Cref{lemma for M splitting} to get a laminate of finite order 
    $$\sum_{j=1}^4 g_j \delta_{G_0^j} + \sum_{j=1}^4b_j \delta_{B_0^j} +\sum_{j=1}^4 c_j\delta_{C_0^j} +\sum_{j=1}^4 d_j\delta_{D_0^j} $$ with barycenter $M$. Let $b\in \R^2$, $\varepsilon,\alpha\in (0,1)$, $s\in (1,\infty)$ and $\Omega$ be a regular domain. Let $\varepsilon_0\ll\varepsilon$. By \Cref{lemma for w from finite laminate}, there exists a piecewise affine Lipschitz map $w_0:\Omega\to \R^2$ with
    \begin{itemize}
        \item $w_0= l_{M,b}$ on $\partial \Omega$;
        \item $\Vert \nabla w_0\Vert_{L^\infty(\Omega)}\leq \max_{S\in \mathcal{S}}|S|\leq m$ where $\mathcal{S}:=\{G_0^j,B_0^j,C_0^j,D_0^j;j=1,2,3,4\}$;
        \item we have
        $$(1-\varepsilon_0)\lambda_S |\Omega| \leq |\{ x\in \Omega:\nabla w_0=S \} |\leq (1+\varepsilon_0) \lambda_S |\Omega|$$
        for each $S\in \mathcal{S}$, where $\lambda_S$ is the coefficient of $\delta_S$ in the laminate.
    \end{itemize}
    In particular, we have for $\Omega_{err,0}:= \{x\in \Omega: \nabla w_0 \notin \mathcal{S}\}$, $|\Omega_{err,0}|<\varepsilon_0 |\Omega|$ and $w_0\in W^{1,1}(\Omega,\R^2)\cap C^\alpha(\overline{\Omega},\R^2)$.

    Fix any $S_0\in \mathcal{S}':=\{B_0^j,C_0^j,D_0^j: j=1,2,3,4\}$. We apply \Cref{prop for staircase laminate for S_0} to get a staircase laminate, $\nu^{\infty}_{S_0}$, supported in $\mathscr{G}$ and with barycenter $S_0$. Moreover,
    \begin{equation}\label{eq nu infny estimate for S0}
        \nu^\infty_{S_0}(\{X: |X|>t\})\leq C |S_0|^qt^{-q}\;\;\forall \;t> 0
    \end{equation}
    for $C=C(q,L)\geq 1$.
    We have $\Omega=\dot{\Omega}\cup \mathcal{N}$ (decomposition corresponding to the piecewise affine map $w_0$, \Cref{defn piecewise affine}). Define
        $$\Omega^{S_0}:= \{x\in \dot{\Omega}: \nabla w_0=S_0\}.$$ 
Since $w_0$ is affine, there exists a decomposition of $\Omega^{S_0}$ into a disjoint union (at most) countably many regular domains, 
$$\Omega^{S_0}=\bigcup_{i} \Omega^{S_0}_i,$$
such that $w_0=l_{S_0,b_i}$ on $\Omega^{S_0}_i$. For each $i$, we apply \Cref{prop from w from nu infinity} to  $\nu^{\infty}_{S_0}$, $\varepsilon_i=\varepsilon_0/2^{i}$ and domain $\Omega^{S_0}_i$ to find a piecewise affine map $w_{S_0,i}\in W^{1,1}(\Omega_i^{S_0},\R^2)\cap C^\alpha(\overline{\Omega_i^{S_0}},\R^2)$  such that $w_{S_0,i}= l_{S_0,b_i}$ on $\partial \Omega_i^{S_0}$ with the following properties:
\begin{itemize}
    \item with $\Omega^{S_0}_{err,i}:=\{x\in \Omega^{S_0}_i: \nabla w_{S_0,i}\notin \mathscr{G}\}$, we have
    $$\int_{\Omega^{S_0}_{err,i}}(1+|\nabla w_{S_0,i}|)^s<\varepsilon_i |\Omega_i^{S_0}|,$$
    \item for each Borel set $E \subset \R^{2\times 2}$
\[
 (1-\eps_i) \nu^\infty_{S_0}(E) \leq \frac{|\{x \in \Omega^{S_0}_i \setminus \Omega_{err,i}^{S_0}: \nabla w_{S_0,i}(x) \in E\}|}{|\Omega_i^{S_0}|} \leq (1+\eps_i) \nu^\infty_{S_0}(E).
\] 
\item and consequently, 
\begin{equation}\label{eq 4.6 c}
    |\{x\in \Omega^{S_0}_i \setminus \Omega_{err,i}^{S_0}: |\nabla w_{S_0,i}(x)|>t\}|\leq M (1+\varepsilon_i)  |S_0|^q |\Omega^{S_0}_i| t^{-q}\;\;\;\forall t>0.
\end{equation}
\end{itemize}
Clearly, from \Cref{prop from w from nu infinity}, we can get $|\Omega^{S_0}_{err,i}|<\varepsilon_i |\Omega_i^{S_0}|$. We can do this for all $S_0\in \mathcal{S}'$(which is a finite set). Then, we define 
$$w:=\begin{cases}
    w_{S_0,i}\;\;\text{in}\;\Omega^{S_0}_i,\\
    w_0\;\;\;\;\;\text{outside }\bigcup_{S_0\in \mathcal{S}'}\Omega^{S_0}.
\end{cases}$$
    By \Cref{lemma gluing argument}, $w\in W^{1,1}(\Omega,\R^2)\cap C^\alpha(\overline{\Omega},\R^2)$. 
     We let 
    $$\Omega_{err}= \Omega_{err,0}\cup \bigcup\limits_{i,S_0}\Omega_{err,i}^{S_0}.$$
    Clearly,  $|\Omega_{err}|<\varepsilon |\Omega|$ for $\varepsilon_0\ll\varepsilon$.
    We also have
    \[
    \begin{split}
        \int_{\Omega_{err}}(1+|\nabla  w|)^s&\leq \int_{\Omega_{err,0}}(1+|\nabla  w_0|)^s +\sum_{i,S_0}\int_{\Omega_{err,i}^{S_0}}(1+|\nabla  w_{S_0,i}|)^s\\
        &\leq (1+\Vert w_0\Vert_{L^{\infty}(\Omega,\R^2)})^s \varepsilon_0|\Omega| +\sum_{i,S_0} \varepsilon_i |\Omega_i^{S_0}|\\
        &\leq \varepsilon |\Omega|
    \end{split}
    \]
    for $\varepsilon_0$ chosen small enough (since $\mathcal{S}'$ is a finite set and  $\Vert w_0\Vert_{L^{\infty}(\Omega)}\leq m$). 

    Now, for any $t\geq m$
    \[
    \begin{split}
        &|\{x \in \Omega\backslash \Omega_{err}: |\nabla w(x)|>t\}|\\&= \underbrace{|\{x \in \Omega\backslash (\cup_{i,S_0}\Omega_{i}^{S_0}\cup \Omega_{err}): |\nabla w_0(x)|>t\}|}_{=0} + \sum_{i,S_0} |\{x \in \Omega_i^{S_0}\backslash \Omega_{err,i}^{S_0}: |\nabla w_{S_0,i}(x)|>t\}\\
        &=\sum_{i,S_0} |\{x \in \Omega_i^{S_0}\backslash \Omega_{err,i}^{S_0}: |\nabla w_{S_0,i}(x)|>t\}\\
        &\leq \sum_{i,S_0} (1+\varepsilon_i)|\Omega_i^{S_0}|\; \nu^{\infty}_{S_0}(\{X: |X|>t\})\\
        & \leq (1+\varepsilon_0)\sum_{S_0} |\Omega^{S_0}|\; \nu^{\infty}_{S_0}(\{X: |X|>t\})\\
        &\aleq_{q,L}  |\Omega| m^qt^{-q} 
    \end{split}
    \]
    where the last step follows from \eqref{eq 4.6 c} and the fact $|S_0|=m$. For $t<m$, we have $m^q t^{-q}\geq 1$ and the result holds trivially. Now, for $t>0$
    \[
    \begin{split}
        |\{x \in \Omega_{err}: |\nabla w(x)|>t\}|&\leq  \int_{\Omega_{err}}\left(\frac{|\nabla w (x)|}{t}\right)^q dx\\
        &\leq \left(\frac{m}{t}\right)^q |\Omega_{err}|\leq 2\varepsilon_0 |\Omega| m^qt^{-q}.
    \end{split}
    \]
    So, for any $t>0$ we have
    \[
    \begin{split}
        |\{x \in \Omega: |\nabla w(x)|>t\}|&\leq |\{x \in \Omega\backslash \Omega_{err}: |\nabla w(x)|>t\}|+|\{x \in \Omega_{err}: |\nabla w(x)|>t\}|\\&\aleq_{q,L} m^q|\Omega| t^{-q}\aleq_{q,L}(1+|M|^q)|\Omega|t^{-q}
    \end{split}
    \]
    where the last inequality follows from \eqref{eq m lesss than 1+M}. Thus, $w$ satisfies all the conditions stated in the theorem and this completes the proof. 
\end{proof}
\begin{proof}[Proof of \Cref{main thm uniqueness}]
    Take any $r\in (1,\bar{q})$. Choose $q$ such that $r<q<\bar{q}$. By \Cref{th matrices space reduced to G}, we see that $\R^{2\times 2}$ is reduced to $\mathcal{G}$ in weak $L^q$. For $\Omega=\B^2$, $M=0$, $b=0$, we apply \Cref{th 4.1} to find  a map $w=(u,v)\in W^{1,1}(\B^2,\R^2)\cap C^\alpha(\overline{\B^2},\R^2)$ such that $w=0$ on $\partial\B^2$, $\nabla w\in \mathscr{G}$ a.e. on $\B^2$ and 
    $$|\{x \in \Omega: |\nabla w(x)|>t\}|\leq C(q,L) |\Omega| t^{-q}\;\;\;\forall t>0.$$
    Now, for $r<q$
    \[
    \begin{split}
        \int_{\Omega} &|\nabla w(x)|^{r}dx= r\int_{0}^{\infty} t^{r-1}|\{x\in \Omega:|\nabla w(x)|>t\}|\;dt\\
        &= r\int_0^1 t^{r-1}|\{x\in \Omega:|\nabla w(x)|>t\}|\;dt +r\int_1^\infty t^{r-1}|\{x\in \Omega:|\nabla w(x)|>t\}|\;dt\\
        &\leq r|\Omega|\int_0^1 t^{r-1}dt + r C(q,L) |\Omega|  \int_1^\infty t^{r-1-q}dt\\
        &\leq C |\Omega|
    \end{split}
    \]
    for some $C=C(q,L)>0$. This implies $w\in W^{1,r}_0(\B^2,\R^2)$.
    
    For a.e. $x\in \B^2$, we have $\nabla w(x)=\begin{pmatrix}
        \nabla u(x)\\\nabla v(x)
    \end{pmatrix}\in \mathscr{G}$. This implies 
    $$A(\nabla u(x)) -J(\nabla v(x))=A(\nabla u(x)) -\nabla^\perp v(x)=0\;\;\text{for a.e. }x\in \B^2.$$
    Hence, $$\div (A(\nabla u))=0\;\;\;\text{in }\B^2.$$  Since $\nabla w\in \mathscr{G}$, $\nabla u\not\equiv 0$ and therefore, $u\not\equiv 0$. Clearly, $u\in W^{1,r}_0(\B^2)$. This proves the theorem.
\end{proof}
\section{\texorpdfstring{Failure of Calder\'on-Zygmund a priori estimate above 2: Proof of \Cref{main thm CZabove2}}{CZ estimate}}\label{section CZabove2}

In this section, we present the proof of \Cref{main thm CZabove2}. Throughout this section, we fix the following parameters.

Fix an even integer $L\geq 4$. Clearly, $0<\frac{L-1}{L}<1.$ Now, we choose a $\theta>0$ such that 
$$\theta\in \left(\left(\frac{L-1}{L}\right)^{2/3}, 1\right).$$ This implies 
\begin{equation}\label{eq: cond on theta and L}
    0< L(1-\theta)<L(1-\theta^{3/2})<1.
\end{equation}
(For example, for $L=4$ we have $\theta=0.9$ satisfying \eqref{eq: cond on theta and L}.)
Define
\begin{equation}\label{eq for p}
    \bar{p}= \frac{L}{L-1}+\frac{1}{1-L(1-\theta)}.
\end{equation}
Since $L>L-1$ and $0<1-L(1-\theta)<1$, we get $\bar{p}>2$. Recall that $A$ is as in  (\ref{eq: for A}) and $h:\R\to \R$ is given  by $h(t)=\sqrt{\theta}\sin(\pi t).$

\textbf{Notation.} Let $\{b_k\}_{k\in \N}$ be as in \Cref{lemma_for b_k} with $\lim_{k\to \infty}b_k=L$. For $k\in \N$, we define 
$$S_k:= \begin{pmatrix}
    b_kk & 0\\
    0& -k
\end{pmatrix},$$
$$G_{k+1}:= \begin{pmatrix}
    b_{k}k & 0\\
    0& -b_{k}k(1+h(b_k k))
\end{pmatrix},$$
and 
$$\Tilde{G}_{k}:= \begin{pmatrix}
    k & 0\\
    0& -k
\end{pmatrix}.
$$
Observe the negative sign in one of the diagonal terms in the matrix $S_k$. This makes the sequence different from the one in \Cref{section CZ}, and this helps to get a threshold exponent greater than 2. Now, we start a decomposition of matrices starting from $S_1$.


  \begin{lemma}\label[lemma]{lemma of splitting CZ above2}
    For $k\in \N$, the following decomposition comes from a laminate of finite order
    $$S_k= \alpha_{k+1}G_{k+1}+\beta_{k+1}\Tilde{G}_{k+1}+\gamma_{k+1}S_{k+1}$$
    where 
    $$\alpha_{k+1}=\frac{1}{(k+1)- b_k k(1+h(b_k k))},$$
    $$\beta_{k+1}= (1-\alpha_{k+1})\left(\frac{b_{k+1}(k+1)-b_k k}{b_{k+1}(k+1)-(k+1)}\right)$$
    and 
    $$\gamma_{k+1}=(1-\alpha_{k+1})\left(1-\frac{b_{k+1}(k+1)-b_k k}{b_{k+1}(k+1)-(k+1)}\right).$$
\end{lemma}
\begin{proof}
    Using the facts $h(b_k k)\in (-\theta,-\theta^{3/2})$ and $2<b_k<L$ for all $k\in \N$ from \Cref{lemma_for b_k} and \eqref{eq: cond on theta and L}, we get
    $$b_k k(1+h(b_k k))<L(1-\theta^{3/2})k<k.$$
    This gives the following elementary splitting of matrices:
    \[
    \begin{split}
        S_k= \begin{pmatrix}
    b_kk & 0\\
    0& -k
\end{pmatrix}= &\alpha_{k+1} \begin{pmatrix}
    b_{k}k & 0\\
    0& -b_{k}k(1+h(b_k k))
\end{pmatrix}+ (1-\alpha_{k+1})\begin{pmatrix}
    b_{k}k & 0\\
    0& -(k+1)
\end{pmatrix}\\
=&\alpha_{k+1} G_{k+1}+ (1-\alpha_{k+1})\begin{pmatrix}
    b_{k}k & 0\\
    0& -(k+1)
\end{pmatrix}
    \end{split}
    \]
    where 
    $$\alpha_{k+1}:=\frac{1}{(k+1)- b_k k(1+h(b_k k))}\in (0,1)$$
    since $(k+1)- b_k k(1+h(b_k k))>1$. Since $b_k>2$, $k+1<b_k k$ for all $k\in \N$. So, we can do the following elementary splitting on the second matrix:
    \[
    \begin{split}
        \begin{pmatrix}
    b_{k}k & 0\\
    0& -(k+1)
\end{pmatrix}&= \Tilde{\beta}_{k+1} \begin{pmatrix}
    k+1 & 0\\
    0& -(k+1)
\end{pmatrix} + (1-\Tilde{\beta}_{k+1}) \begin{pmatrix}
    b_{k+1}(k+1) & 0\\
    0& -(k+1)
\end{pmatrix}\\
&= \Tilde{\beta}_{k+1} \Tilde{G}_{k+1} + (1-\Tilde{\beta}_{k+1})S_{k+1}
    \end{split}
    \]
    where 
    $$\Tilde{\beta}_{k+1}:=\frac{b_{k+1}(k+1)-b_k k}{b_{k+1}(k+1)-(k+1)} \in (0,1).$$
    These two splittings give the desired result.
\end{proof}
\begin{lemma}\label[lemma]{lemma for beta_n estimate CZ above2}
There exist constants $C_1,C_2>0$ (depending only on $L$ and $\theta$) such that the following holds: For $k\in \N$, define 
    \begin{equation}
        \gamma_{k+1}:= \left(1- \frac{1}{k+1- b_k k(1+h(b_k k))}\right)\left(1-\frac{b_{k+1}(k+1)-b_k k}{b_{k+1}(k+1)-(k+1)}\right).
    \end{equation}
Then
    $$C_1 N^{-\bar{p}}\leq \prod_{k=1}^N \gamma_{k+1}\leq C_2 N^{-\bar{p}} $$
    for all $N\geq 1$. 
\end{lemma}
\begin{proof}
We have
$$\gamma_{k+1}=\left(1- \frac{1}{k+1- b_k k(1+h(b_k k))}\right)\left(1-\frac{b_{k+1}(k+1)-b_k k}{b_{k+1}(k+1)-(k+1)}\right).$$
Then
\begin{equation}\label{eq -log gamma CZabove2}
\begin{split}
    -\log \gamma_{k+1} =&-\log \left(1-\frac{1}{(k+1)- b_k k(1+h(b_k k))}\right)\\ &-\log \left(1- \frac{b_{k+1}(k+1)-b_k k}{b_{k+1}(k+1)-(k+1)}\right).
\end{split}
\end{equation}
Let 
\[
\begin{split}
    \alpha_k&:=\frac{1}{(k+1)-b_k k(1+h(b_k k))}\\
    &= \frac{1}{k(1-b_k (1+h(b_k k)))\left(1+\frac{1}{k(1-b_k(1+h(b_k k)))}\right)}\\
    &=  \frac{1}{k(1-b_k (1+h(b_k k)))} +O\left(\frac{1}{k^2}\right).
\end{split}
\]
In the above step, we used the expansion
\[
\frac{1}{1+\delta}
= 1 - \delta + \delta^{2} - \delta^{3} + \delta^4-...
\]
for $0<\delta<1$. Since $b_k \uparrow L$, $h(b_k k)\downarrow -\theta$ and $L-b_k=O(\frac{1}{k})$ (from \Cref{lemma_for b_k}), we get 
\begin{equation}\label{eq: order for alpha k}
   \alpha_k= \frac{1}{k}\left(\frac{1}{1-L(1-\theta)}\right)+O\left(\frac{1}{k^2}\right). 
\end{equation}
Similarly, let 
\[
\begin{split}
\Tilde{\alpha}_k&:=\frac{b_{k+1}(k+1)-b_k k}{b_{k+1}(k+1)-(k+1)}\\&= \frac{b_{k+1}+k(b_{k+1}-b_k)}{k(b_{k+1}-1)\left(1+\frac{1}{k}\right)}\\
&= \frac{b_{k+1}}{k(b_{k+1}-1)} + \frac{k(b_{k+1}-b_k)}{k(b_{k+1}-1)} +O\left(\frac{1}{k^2}\right).
\end{split}
\]
Since $b_{k+1}-b_k=O(\frac{1}{k^2})$ and $L-b_k=O(\frac{1}{k})$, we get
\begin{equation}\label{eq: order for alpha tilde k}
    \Tilde{\alpha}_k=\frac{1}{k}\left(\frac{L}{L-1}\right)+O\left(\frac{1}{k^2}\right).
\end{equation}
Now, using the expansion 
$$-\log (1-\varepsilon)=\varepsilon+\frac{\varepsilon^2}{2}+\frac{\varepsilon^3}{3}+...$$
for $0<\varepsilon<1$, in \eqref{eq -log gamma CZabove2} we get
$$-\log \gamma_{k+1}=\alpha_k+\Tilde{\alpha}_k +O(\alpha_k^2)+O(\Tilde{\alpha}_k^2).$$
Using \eqref{eq: order for alpha k} and \eqref{eq: order for alpha tilde k}, we have
\[
\begin{split}
    -\log \gamma_{k+1}&= \frac{1}{k}\left(\frac{1}{1-L(1-\theta)}+\frac{L}{L-1}\right)+O\left(\frac{1}{k^2}\right)\\
    &=\frac{1}{k}\bar{p}+O\left(\frac{1}{k^2}\right).
\end{split}
\]
Then
$$\sum_{k=1}^N \log \gamma_{k+1} = -\bar{p}\sum_{k=1}^N\frac{1}{k}+O(1)=-\bar{p}\log N +O(1).$$
So, 
\[
\prod_{k=1}^N \gamma_{k+1}
= \exp\!\left(\sum_{k=1}^N \log \gamma_{k+1}\right)
= \exp\!\left(-\bar{p} \log N + O(1)\right)
= N^{-\bar{p}}\exp(O(1)).
\]
Thus, there exist constants $C_1,C_2>0$ (depending only on $L$ and $\theta$) such that 
$$C_1 N^{-\bar{p}}\leq\prod_{k=1}^N \gamma_{k+1}\leq C_2 N^{-\bar{p}}.$$
\end{proof}
\begin{lemma}\label[lemma]{lemma for norm on good set}
    For $k\in \N$, let $\alpha_{k+1}$ and $\gamma_{k+1}$ be as in \Cref{lemma of splitting CZ above2}. Define $\Gamma_0:=1$, 
    $$\Gamma_k:=\prod_{j=1}^{k} \gamma_{j+1}.$$ Then there exists a constant $C>0$ (depending only on $L$ and $\theta$) such that
    $$ \sum_{k=1}^N \Gamma_{k-1}\alpha_{k+1}(b_k k)^{\bar{p}}\geq C\log N - O(1).$$
\end{lemma}
\begin{proof}
    From \Cref{lemma for beta_n estimate CZ above2}, there exists constant $C_1,C_2>0$ (depending only on $L$ and $\theta $) such that 
    $$C_1 N^{-\bar{p}}\leq\Gamma_k\geq C_2 k^{-\bar{p}}.$$
    From \eqref{eq: order for alpha k}, we have
    $$\alpha_{k+1}= \frac{1}{k}\left(\frac{1}{1-L(1-\theta)}\right)+O\left(\frac{1}{k^2}\right).$$
    So, we have
    \[
    \begin{split}
         \sum_{k=1}^N \Gamma_{k-1}\alpha_{k+1}(b_k k)^{\bar{p}} 
        \ageq_{L,\theta} \sum_{k=1}^N \frac{1}{k} -O\left(\frac{1}{k^2}\right)
        \geq C \log N -O(1)
    \end{split}
    \]
    where $C>0$ is a constant depending only on $L$ and $\theta$.
\end{proof}
Combining \Cref{lemma of splitting CZ above2} and \Cref{lemma for beta_n estimate CZ above2}, we have 
\begin{proposition}\label[proposition]{prop_laminate CZabove2}
    Let $\alpha_{k+1}$, $\beta_{k+1}$ and $\gamma_{k+1}$ be as in \Cref{lemma of splitting CZ above2}. 
    Then for integers $N\geq 1$, there exist sequences $\{\overline{\alpha}_{k}\}_{k=1}^N$, $\{\overline{\beta}_{k}\}_{k=1}^N$ with values in $(0,1)$ and a number $\Gamma_N\in (0,1)$ such that
    \begin{itemize}
        \item[(a)] the following is a laminate of finite order with barycenter $S_1\in \R^{2\times 2}$
        $$\sum_{k=1}^N(\overline{\alpha}_{k} \delta_{G_{k}}+\overline{\beta}_{k} \delta_{\Tilde{G}_k})+\Gamma_N \delta_{S_N}$$
        where $\Gamma_0:=1$,
        $\Gamma_k:=\prod_{j=1}^k \gamma_{j+1}$, $\overline{\alpha}_{k}:= \Gamma_{k-1} \alpha_{k+1}$, $\overline{\beta}_{k}:= \Gamma_{k-1}\beta_{k+1}$;
        \item[(b)] and we have
        $$C_1N^{-\bar{p}}\leq \Gamma_N \leq C_2 N^{-\bar{p}}$$
        for some constants $C_1,C_2>0$ depending only on $L$ and $\theta$.
    \end{itemize}
\end{proposition}
Now, we present the proof of \Cref{main thm CZabove2}.
\begin{proof}[Proof of \Cref{main thm CZabove2}]
    Let $\theta$, $L$ and $\bar{p}$ be as in \eqref{eq for p} and $\{b_k\}_{k\in \N}$ be as in \Cref{lemma_for b_k}. Define the map $l:\R^2\to \R$ by $l(x_1,x_2):=b_1 x_1$.
    
    Fix $\Lambda>0$. Let $N\gg1$ (to be chosen depending on $\Lambda$). We apply \Cref{prop_laminate CZabove2} to get the laminate of finite order with center of mass $S_1$. Let $0<\varepsilon< N^{-\bar{p}}$. Next, we apply \Cref{lemma for w from finite laminate} to obtain a piecewise affine Lipschitz map 
    $$w=\begin{pmatrix}
        u\\v
    \end{pmatrix}: \B^2\to \R^2$$
    such that $w(x)=S_1 x$ for $x\in \partial \B^2$ and $\Vert \nabla w\Vert_{L^{\infty}(\B^2)}\aleq_L N$. This gives $u=l$ on $\partial \B^2$ and clearly, we have $u,v\in W^{1,\infty}(\B^2)$ since $w$ is a piecewise affine Lipschitz map. Note that the map $w=(u,v)$ depends on $N$. We claim that 
    \begin{equation}\label{eq claim CZabove2}
        \int_{\B^2}|\nabla u|^{\bar{p}}>\Lambda \left(1+\int_{\B^2}|A(\nabla u)-\nabla^{\perp }v|^{\bar{p}}\right)
    \end{equation}
    for sufficiently large $N$. Indeed, denote for $S\in \R^{2\times 2}$
    $$\Omega_S:= \left\{x\in \B^2:\;\nabla w (x)=S\right\}.$$ Let $G_k$, $\Tilde{G}_k$ and $S_k$ be as in \Cref{prop_laminate CZabove2}. Define
    $$\Omega_{err}:=\Omega\backslash \left(\Omega_{S_N}\dot{\cup} \;\dot{\bigcup}_{k=1}^N\Omega_{G_k} \dot{\cup}\Omega_{\Tilde{G}_k}\right)$$
    where $\dot{\cup}$ denotes a disjoint union of sets. (For simplicity of notation, we ignore the null set $\mathcal{N}$ coming from \Cref{defn piecewise affine}.) From \Cref{lemma for w from finite laminate}, 
    we get $$|\Omega_{err}|\leq \varepsilon |\B^2|$$
    and 
    \begin{equation}\label{eq: f bound on err set CZabove2}
        \int_{\Omega_{err}}|A(\nabla u)-\nabla^\perp v|^{\bar{p}}\aleq_{L,\theta} \varepsilon \Vert \nabla w\Vert^{\bar{p}}_{L^{\infty}(\B^2)} |\B^2|\aleq_{L,\theta} \varepsilon |\B^2| N^{\bar{p}}\aleq _{L,\theta} 1
    \end{equation}
    since $\varepsilon<N^{-\bar{p}}$. Also, we have
    \begin{equation}\label{eq: measure of Omega g_k and s_k}
            \begin{split}
        |\Omega_{ S_N}|&\aleq  \Gamma_N|\B^2|\\ |\Omega_{ G_k}|&\ageq  \overline{\alpha}_k|\B^2|\;\;\forall k
    \end{split}
    \end{equation}
    where $\overline{\alpha}_k$ and $\Gamma_N$ are as in \Cref{prop_laminate CZabove2}.
    
    Fix any $k\in \N$.  For any $x\in \Omega_{G_k}$ or $x\in \Omega_{\Tilde{G}_k}$  ,
    \begin{equation}\label{en: f bound on omega s CZabove2}
        A(\nabla u(x))-\nabla^{\perp} v(x)=0.
    \end{equation}
    From \eqref{eq: measure of Omega g_k and s_k} and \Cref{lemma for beta_n estimate CZ above2}, we get
    \begin{equation}\label{eq f bound on omega An CZabove2}
        \int_{\Omega_{ S_N}}|A(\nabla u(x))-\nabla^{\perp} v(x)|^{\bar{p}} \aleq_{L,\theta} \Gamma_N N^{\bar{p}} \aleq_{L,\theta} 1.
    \end{equation}
    Hence, using (\ref{eq: f bound on err set CZabove2}), (\ref{en: f bound on omega s CZabove2}) and (\ref{eq f bound on omega An CZabove2}) we have
    \begin{equation*}
        \int_{\B^2} |A(\nabla u)-\nabla^\perp v|^{\bar{p}}\aleq_{L,\theta}  1.
    \end{equation*} 
    Thus, we have
    \begin{equation}\label{eq: RHS CZ above2}
        \Lambda \left(1+\int_{\B^2}|A(\nabla u)-\nabla^{\perp }v|^{\bar{p}}\right)
        \leq C_1\Lambda 
    \end{equation}
where $C_1>0$ depends only on $L$ and $\theta$ and is independent of $N$. On the other hand, using \eqref{eq: measure of Omega g_k and s_k}, we get
\[ 
\begin{split}
    \int_{\B^2}|\nabla u|^{\bar{p}} \geq \sum_{k=1}^N\int_{\Omega_{G_k}}|\nabla u|^{\bar{p}}
    \ageq_{L,\theta} \sum _{k=1}^N \Gamma_{k-1}\alpha_{k+1} (b_k k)^{\bar{p}} .
\end{split}
\]
Applying \Cref{lemma for norm on good set}, we get
\begin{equation}\label{eq LHS of CZ above2}
    \int_{\B^2}|\nabla u|^{\bar{p}} \geq C \log N -O(1)
\end{equation}
where $C>0$ depends only on $L$ and $\theta$. Since the constants $C,C_1>0$ are independent of $N$ and $\varepsilon$, choose a large $N$ and a small $\varepsilon$ to get
$$\Lambda \left(1+\int_{\B^2}|A(\nabla u)-\nabla^{\perp }v|^{\bar{p}}\right)
        \leq C_1\Lambda < C \log N -O(1)\leq \int_{\B^2}|\nabla u|^{\bar{p}}$$
by combining \eqref{eq LHS of CZ above2} and \eqref{eq: RHS CZ above2}. This proves the claim.

Now, take $F:=A(\nabla u)-\nabla^\perp v$. Then, we have 
\[
\begin{cases}
\div (A(\nabla u)) = \div F \quad& \text{in $\B^2$},\\
u=l \quad &\text{on $\partial \B^2$}.
\end{cases}
\]
We have $u\in W^{1,\infty}(\B^2)$ and $F\in L^\infty(\B^2,\R^2)$. But 
$$\int_{\B^2}|\nabla u|^{\bar{p}}>\Lambda\left(1+ \int_{\B^2}|F|^{\bar{p}}\right).$$ This completes the proof.
\end{proof}





\bibliographystyle{abbrv}
\bibliography{bib}

@ARTICLE{MS2024,
       author = {{Mazowiecka}, Katarzyna and {Schikorra}, Armin},
        title = "{Nowhere smooth critical points of polyconvex functionals in arbitrary dimension}",
      journal = {arXiv e-prints},
         year = 2024,
        month = may,
          eid = {arXiv:2405.17084},
        pages = {arXiv:2405.17084},
          doi = {10.48550/arXiv.2405.17084},
archivePrefix = {arXiv},
       eprint = {2405.17084},
 primaryClass = {math.AP},
       adsurl = {https://ui.adsabs.harvard.edu/abs/2024arXiv240517084M}
}

@ARTICLE{KMSX24,
    AUTHOR = {Kleiner, Bruce and M\"uller, Stefan and Sz\'ekelyhidi, Jr.,
              L\'aszl\'o{} and Xie, Xiangdong},
     TITLE = {Rigidity of {E}uclidean product structure: breakdown for low
              {S}obolev exponents},
   JOURNAL = {Commun. Pure Appl. Anal.},
  FJOURNAL = {Communications on Pure and Applied Analysis},
    VOLUME = {23},
      YEAR = {2024},
    NUMBER = {10},
     PAGES = {1569--1607},
      ISSN = {1534-0392,1553-5258},
   MRCLASS = {35R70 (30C62 30C65 35A35 46E35)},
  MRNUMBER = {4799456},
MRREVIEWER = {Cintia\ Pacchiano Camacho},
       DOI = {10.3934/cpaa.2024029},
       URL = {https://doi.org/10.3934/cpaa.2024029},
}

@article {BDS,
    AUTHOR = {Bul\'i\v{c}ek, Miroslav and Diening, Lars and Schwarzacher,
              Sebastian},
     TITLE = {Existence, uniqueness and optimal regularity results for very
              weak solutions to nonlinear elliptic systems},
   JOURNAL = {Anal. PDE},
  FJOURNAL = {Analysis \& PDE},
    VOLUME = {9},
      YEAR = {2016},
    NUMBER = {5},
     PAGES = {1115--1151},
      ISSN = {2157-5045,1948-206X},
   MRCLASS = {35J57 (35A01 35A02 35B65 35D30 35J62)},
  MRNUMBER = {3531368},
MRREVIEWER = {Matthias\ K\"ohne},
       DOI = {10.2140/apde.2016.9.1115},
       URL = {https://doi.org/10.2140/apde.2016.9.1115},
}

@article {MS03,
    AUTHOR = {M\"{u}ller, S. and \v{S}ver\'{a}k, V.},
     TITLE = {Convex integration for {L}ipschitz mappings and
              counterexamples to regularity},
   JOURNAL = {Ann. of Math. (2)},
  FJOURNAL = {Annals of Mathematics. Second Series},
    VOLUME = {157},
      YEAR = {2003},
    NUMBER = {3},
     PAGES = {715--742},
      ISSN = {0003-486X},
   MRCLASS = {35D10 (35J45 35J50 49J10 49N60)},
  MRNUMBER = {1983780},
MRREVIEWER = {John M. Ball},
       DOI = {10.4007/annals.2003.157.715},
       URL = {https://doi.org/10.4007/annals.2003.157.715},
}

@article {AFS08,
    AUTHOR = {Astala, Kari and Faraco, Daniel and Sz\'ekelyhidi, Jr.,
              L\'aszl\'o},
     TITLE = {Convex integration and the {$L^p$} theory of elliptic
              equations},
   JOURNAL = {Ann. Sc. Norm. Super. Pisa Cl. Sci. (5)},
  FJOURNAL = {Annali della Scuola Normale Superiore di Pisa. Classe di
              Scienze. Serie V},
    VOLUME = {7},
      YEAR = {2008},
    NUMBER = {1},
     PAGES = {1--50},
      ISSN = {0391-173X,2036-2145},
   MRCLASS = {35J25 (30C62 35B65 35D10)},
  MRNUMBER = {2413671},
MRREVIEWER = {Leonid\ V.\ Kovalev},
}

@ARTICLE{Armin24,
       author = {{Schikorra}, Armin},
        title = "{Failure of $L^r$-Calder{\'o}n-Zygmund estimates for the p-Laplace equation for small $r$}",
      journal = {arXiv e-prints},
         year = 2024,
        month = aug,
          eid = {arXiv:2408.03546},
        pages = {arXiv:2408.03546},
          doi = {10.48550/arXiv.2408.03546},
archivePrefix = {arXiv},
       eprint = {2408.03546},
 primaryClass = {math.AP},
       adsurl = {https://ui.adsabs.harvard.edu/abs/2024arXiv240803546S}
}

@article {IS94,
    AUTHOR = {Iwaniec, T. and Sbordone, C.},
     TITLE = {Weak minima of variational integrals},
   JOURNAL = {J. Reine Angew. Math.},
  FJOURNAL = {Journal f\"ur die Reine und Angewandte Mathematik. [Crelle's
              Journal]},
    VOLUME = {454},
      YEAR = {1994},
     PAGES = {143--161},
      ISSN = {0075-4102,1435-5345},
   MRCLASS = {49K10 (35J99 49N60)},
  MRNUMBER = {1288682},
MRREVIEWER = {Francesco\ Ferro},
       DOI = {10.1515/crll.1994.454.143},
       URL = {https://doi.org/10.1515/crll.1994.454.143},
}

@article {I83,
    AUTHOR = {Iwaniec, Tadeusz},
     TITLE = {Projections onto gradient fields and {$L\sp{p}$}-estimates for
              degenerated elliptic operators},
   JOURNAL = {Studia Math.},
  FJOURNAL = {Polska Akademia Nauk. Instytut Matematyczny. Studia
              Mathematica},
    VOLUME = {75},
      YEAR = {1983},
    NUMBER = {3},
     PAGES = {293--312},
      ISSN = {0039-3223,1730-6337},
   MRCLASS = {46E35 (30C60 35J70)},
  MRNUMBER = {722254},
MRREVIEWER = {S.\ K.\ Vodop\cprime yanov},
       DOI = {10.4064/sm-75-3-293-312},
       URL = {https://doi.org/10.4064/sm-75-3-293-312},
}

@article {FMCO18,
    AUTHOR = {Faraco, Daniel and Mora-Corral, Carlos and Oliva, Marcos},
     TITLE = {Sobolev homeomorphisms with gradients of low rank via
              laminates},
   JOURNAL = {Adv. Calc. Var.},
  FJOURNAL = {Advances in Calculus of Variations},
    VOLUME = {11},
      YEAR = {2018},
    NUMBER = {2},
     PAGES = {111--138},
      ISSN = {1864-8258,1864-8266},
   MRCLASS = {46E35 (26B25 26B35)},
  MRNUMBER = {3778577},
MRREVIEWER = {Alexander\ D.\ Ukhlov},
       DOI = {10.1515/acv-2016-0009},
       URL = {https://doi.org/10.1515/acv-2016-0009},
}

@article{Pedregal_1993, title={Laminates and microstructure}, volume={4}, DOI={10.1017/S0956792500001030}, number={2}, journal={European Journal of Applied Mathematics}, author={Pedregal, Pablo}, year={1993}, pages={121–149}}

@article {Dac85,
    AUTHOR = {Dacorogna, B.},
     TITLE = {Remarques sur les notions de polyconvexit\'e,
              quasi-convexit\'e{} et convexit\'e{} de rang {$1$}},
   JOURNAL = {J. Math. Pures Appl. (9)},
  FJOURNAL = {Journal de Math\'ematiques Pures et Appliqu\'ees. Neuvi\`eme
              S\'erie},
    VOLUME = {64},
      YEAR = {1985},
    NUMBER = {4},
     PAGES = {403--438},
      ISSN = {0021-7824},
   MRCLASS = {26B25 (49A50)},
  MRNUMBER = {839729},
MRREVIEWER = {Carlo\ Bardaro},
}

@article {Cel93,
    AUTHOR = {Cellina, Arrigo},
     TITLE = {On minima of a functional of the gradient: necessary
              conditions},
   JOURNAL = {Nonlinear Anal.},
  FJOURNAL = {Nonlinear Analysis. Theory, Methods \& Applications. An
              International Multidisciplinary Journal},
    VOLUME = {20},
      YEAR = {1993},
    NUMBER = {4},
     PAGES = {337--341},
      ISSN = {0362-546X,1873-5215},
   MRCLASS = {49K10 (49J40)},
  MRNUMBER = {1206422},
MRREVIEWER = {Francesco\ Ferro},
       DOI = {10.1016/0362-546X(93)90137-H},
       URL = {https://doi.org/10.1016/0362-546X(93)90137-H},
}

@book {DM99,
    AUTHOR = {Dacorogna, Bernard and Marcellini, Paolo},
     TITLE = {Implicit partial differential equations},
    SERIES = {Progress in Nonlinear Differential Equations and their
              Applications},
    VOLUME = {37},
 PUBLISHER = {Birkh\"auser Boston, Inc., Boston, MA},
      YEAR = {1999},
     PAGES = {xiv+273},
      ISBN = {0-8176-4121-1},
   MRCLASS = {35A25 (35F20 35G20)},
  MRNUMBER = {1702252},
MRREVIEWER = {Jolanta\ Przybycin},
       DOI = {10.1007/978-1-4612-1562-2},
       URL = {https://doi.org/10.1007/978-1-4612-1562-2},
}

@article {syc01,
    AUTHOR = {Sychev, M. A.},
     TITLE = {Comparing two methods of resolving homogeneous differential
              inclusions},
   JOURNAL = {Calc. Var. Partial Differential Equations},
  FJOURNAL = {Calculus of Variations and Partial Differential Equations},
    VOLUME = {13},
      YEAR = {2001},
    NUMBER = {2},
     PAGES = {213--229},
      ISSN = {0944-2669,1432-0835},
   MRCLASS = {35R70 (49K24)},
  MRNUMBER = {1861098},
MRREVIEWER = {Tzanko\ D.\ Donchev},
       DOI = {10.1007/PL00009929},
       URL = {https://doi.org/10.1007/PL00009929},
}

@incollection {KSM03,
    AUTHOR = {Kirchheim, Bernd and M\"uller, Stefan and \v{S}ver\'ak,
              Vladim\'ir},
     TITLE = {Studying nonlinear pde by geometry in matrix space},
 BOOKTITLE = {Geometric analysis and nonlinear partial differential
              equations},
     PAGES = {347--395},
 PUBLISHER = {Springer, Berlin},
      YEAR = {2003},
      ISBN = {3-540-44051-8},
   MRCLASS = {35J60 (35K55 49Q20)},
  MRNUMBER = {2008346},
MRREVIEWER = {J.\ Chrastina},
}

@article {DMP08,
    AUTHOR = {Dacorogna, B. and Marcellini, P. and Paolini, E.},
     TITLE = {Lipschitz-continuous local isometric immersions: rigid maps
              and origami},
   JOURNAL = {J. Math. Pures Appl. (9)},
  FJOURNAL = {Journal de Math\'ematiques Pures et Appliqu\'ees. Neuvi\`eme
              S\'erie},
    VOLUME = {90},
      YEAR = {2008},
    NUMBER = {1},
     PAGES = {66--81},
      ISSN = {0021-7824},
   MRCLASS = {26B35 (26B12 35J25)},
  MRNUMBER = {2435215},
MRREVIEWER = {Steven\ George\ Krantz},
       DOI = {10.1016/j.matpur.2008.02.011},
       URL = {https://doi.org/10.1016/j.matpur.2008.02.011},
}

@article {Pom10,
    AUTHOR = {Pompe, Waldemar},
     TITLE = {Explicit construction of piecewise affine mappings with
              constraints},
   JOURNAL = {Bull. Pol. Acad. Sci. Math.},
  FJOURNAL = {Bulletin of the Polish Academy of Sciences. Mathematics},
    VOLUME = {58},
      YEAR = {2010},
    NUMBER = {3},
     PAGES = {209--220},
      ISSN = {0239-7269,1732-8985},
   MRCLASS = {26B25 (49K21)},
  MRNUMBER = {2771571},
MRREVIEWER = {Peter\ Recht},
       DOI = {10.4064/ba58-3-4},
       URL = {https://doi.org/10.4064/ba58-3-4},
}

@book{kir03,
  title={Rigidity and Geometry of Microstructures},
  author={Kirchheim, B.},
  series={Lecture notes},
  url={https://books.google.com/books?id=D0f1GwAACAAJ},
  year={2003},
  publisher={Max-Planck-Inst. f{\"u}r Mathematik in den Naturwiss.}
}

@article {CFM05,
    AUTHOR = {Conti, Sergio and Faraco, Daniel and Maggi, Francesco},
     TITLE = {A new approach to counterexamples to {$L^1$} estimates:
              {K}orn's inequality, geometric rigidity, and regularity for
              gradients of separately convex functions},
   JOURNAL = {Arch. Ration. Mech. Anal.},
  FJOURNAL = {Archive for Rational Mechanics and Analysis},
    VOLUME = {175},
      YEAR = {2005},
    NUMBER = {2},
     PAGES = {287--300},
      ISSN = {0003-9527,1432-0673},
   MRCLASS = {74B20 (26D15 35J50 49N60 74G65)},
  MRNUMBER = {2118479},
MRREVIEWER = {John\ M.\ Ball},
       DOI = {10.1007/s00205-004-0350-5},
       URL = {https://doi.org/10.1007/s00205-004-0350-5},
}

@article {CFMM05,
    AUTHOR = {Conti, S. and Faraco, D. and Maggi, F. and M\"uller, S.},
     TITLE = {Rank-one convex functions on {$2\times 2$} symmetric matrices
              and laminates on rank-three lines},
   JOURNAL = {Calc. Var. Partial Differential Equations},
  FJOURNAL = {Calculus of Variations and Partial Differential Equations},
    VOLUME = {24},
      YEAR = {2005},
    NUMBER = {4},
     PAGES = {479--493},
      ISSN = {0944-2669,1432-0835},
   MRCLASS = {26B25 (26D15 35J50 49N60)},
  MRNUMBER = {2180863},
MRREVIEWER = {Mikil\ D.\ Foss},
       DOI = {10.1007/s00526-005-0343-8},
       URL = {https://doi.org/10.1007/s00526-005-0343-8},
}

@article {bsv13,
    AUTHOR = {Boros, Nicholas and Sz\'ekelyhidi, Jr., L\'aszl\'o{} and
              Volberg, Alexander},
     TITLE = {Laminates meet {B}urkholder functions},
   JOURNAL = {J. Math. Pures Appl. (9)},
  FJOURNAL = {Journal de Math\'ematiques Pures et Appliqu\'ees. Neuvi\`eme
              S\'erie},
    VOLUME = {100},
      YEAR = {2013},
    NUMBER = {5},
     PAGES = {687--700},
      ISSN = {0021-7824,1776-3371},
   MRCLASS = {42B20 (47A55 60G46)},
  MRNUMBER = {3115829},
MRREVIEWER = {Pierre\ Portal},
       DOI = {10.1016/j.matpur.2013.01.017},
       URL = {https://doi.org/10.1016/j.matpur.2013.01.017},
}

@misc{fls21,
author = {Faraco, Daniel and Lindberg, Sauli and Székelyhidi, Jr},
year = {2021},
month = {09},
pages = {},
title = {Magnetic helicity, weak solutions and relaxation of ideal MHD},
doi = {10.48550/arXiv.2109.09106}
}

@article{columbostudent,
author = {Johansson, Carl},
year = {2023},
month = {12},
pages = {},
title = {Wild solutions to scalar Euler-Lagrange equations},
journal = {Transactions of the American Mathematical Society},
doi = {10.1090/tran/9090}
}

@article {colombo2022nonclassicalsolutionsplaplaceequation,
    AUTHOR = {Colombo, Maria and Tione, Riccardo},
     TITLE = {Non-classical solutions of the {$p$}-{L}aplace equation},
   JOURNAL = {J. Eur. Math. Soc. (JEMS)},
  FJOURNAL = {Journal of the European Mathematical Society (JEMS)},
    VOLUME = {27},
      YEAR = {2025},
    NUMBER = {12},
     PAGES = {4845--4890},
      ISSN = {1435-9855,1435-9863},
   MRCLASS = {35D30 (35J15 35J60 35J70)},
  MRNUMBER = {4956978},
       DOI = {10.4171/jems/1462},
       URL = {https://doi.org/10.4171/jems/1462},
}

@article {faracomilton,
    AUTHOR = {Faraco, Daniel},
     TITLE = {Milton's conjecture on the regularity of solutions to
              isotropic equations},
   JOURNAL = {Ann. Inst. H. Poincar\'e{} C Anal. Non Lin\'eaire},
  FJOURNAL = {Annales de l'Institut Henri Poincar\'e{} C. Analyse Non
              Lin\'eaire},
    VOLUME = {20},
      YEAR = {2003},
    NUMBER = {5},
     PAGES = {889--909},
      ISSN = {0294-1449,1873-1430},
   MRCLASS = {35J15 (30C62 30G20 49J45 60B10)},
  MRNUMBER = {1995506},
MRREVIEWER = {Andrei\ B.\ Bogatyr\"ev},
       DOI = {10.1016/S0294-1449(03)00014-3},
       URL = {https://doi.org/10.1016/S0294-1449(03)00014-3},
}

\end{document}